\documentclass[11pt, a4paper]{amsart}
\usepackage[english]{babel}
\usepackage{amssymb,amsmath,amsthm,mathrsfs,bm}
\usepackage{color}
\usepackage[margin=1in]{geometry}
\usepackage{stmaryrd}

\newtheorem{theorem}[subsection]{Theorem}
\newtheorem{conjecture}[subsection]{Conjecture}

\newtheorem{lemma}[subsection]{Lemma}
\newtheorem{remark}[subsection]{Remark}
\newtheorem{proposition}[subsection]{Proposition}
\newtheorem{corollary}[subsection]{Corollary}

\newtheorem*{assumption*}{Assumption}
\newtheorem*{claim*}{Claim}
\newtheorem*{theorem*}{Theorem}
\newenvironment{example}{\smallskip \refstepcounter{subsection}
{\bf Example \thesubsection.}}

\def\bal{\begin{aligned}}
\def\eal{\end{aligned}}
\def\be{\begin{equation}}
\def\ee{\end{equation}}
\def\bcs{\begin{cases}}
\def\ecs{\end{cases}}
\def\={\;=\;}
\def\+{\,+\,}
\def\-{\,-\,}
\def\b{\backslash}
\def\C{{\mathbb C}}
\def\Z{{\mathbb Z}}

\def\Q{{\mathbb Q}}

\def\pow#1{\llbracket #1\rrbracket}

\def\v#1{{\bf #1}}

\def\is{\equiv}

\def\mod#1{({\rm mod}\ #1)}

\def\ceil#1{\lceil #1\rceil}
\def\floor#1{\lfloor #1\rfloor}

\title{Supercongruences using modular forms}
\author{Frits Beukers}

\address{Utrecht University }
\email{f.beukers@uu.nl}

\date{\today}
\begin{document}
\maketitle

\begin{abstract}
Let $F(t)$ be the generating series of some combinatorially interesting numbers,
like the Ap\'ery numbers.
For any prime $p$ we let $F_p(t)$ be the sum of the terms of degree $<p$ (the $p$-truncation
of $F$). It is a classical fact that for many $p$ and algebraic arguments $\alpha$
we have the congruence $F_p(\alpha)\is\theta\mod{p}$, where $\theta$ is a zero of the
characteristic $p$ $\zeta$-function of an algebraic variety associated to $F$ and $\alpha$.
 
Surprisingly, very often these congruences turn out to hold modulo
$p^2$ or even $p^3$. We call such congruences supercongruences and in the past 15 years
an abundance of them have been discovered. In this paper we show that a large proportion
of them can be explained by the use of modular functions and forms.
\end{abstract}

\section{Introduction}
A motivating example for this paper is the so-called \emph{Legendre example}.
Let $p$ be an odd prime and let $\Z_p$ be the ring of $p$-adic integers. Define
the \emph{Hasse polynomial}
$$
H_p(t):=\sum_{k=0}^{p-1}\binom{-1/2}{k}^2t^k.
$$
It is well-know, see \cite{Deu41}, that for any $t_0\in\Z_p$ we have
the congruence $H_p(t_0)\is a_p\mod{p}$ where $a_p=p-N_p$ and
$N_p$ is the number of solutions $x,y\mod{p}$ to the congruence
$$
y^2\is x(x-1)(x-t_0)\mod{p}.
$$
The family of elliptic curves $y^2=x(x-1)(x-t)$ parametrized by $t$ is known as
the \emph{Legendre family}. It is the toy example of many books on algebraic geometry. 
When $t_0=0,1$ we have $a_p=1$ and $a_p=(-1)^{(p-1)/2}$ respectively.
Let us assume $t_0\not\is0,1\mod{p}$.
A famous result of Hasse states that $|a_p|<2\sqrt{p}$, which implies that the
local $\zeta$-function $X^2-a_pX+p$ has two complex zeros of absolute value $\sqrt{p}$. 
Suppose $a_p\not\is0\mod{p}$. Then $X^2-a_pX+p$ has a unique $p$-adic zero $\vartheta_p\in\Z_p$
with $|\vartheta_p|_p=1$, which is called the \emph{($p$-adic) unit root} of $X^2-a_pX+p$. In this case
our congruence can be rephrased as $H_p(t_0)\is\vartheta_p\mod{p}$. To avoid cluttering of the notation
we have suppressed the dependence of $a_p$ and $\vartheta_p$ on $t_0$ in our notation.

It turns out that for certain values of $t_0$ we get $H_p(t_0)\is\vartheta_p\mod{p^2}$
for infinitely many primes $p$. This is a congruence modulo $p^2$ instead of $p$. Here is an example,

\begin{theorem}\label{legendre-example}
We have
\begin{itemize}
\item[--]
For all odd primes $p$: $H_p(1)\is(-1)^{(p-1)/2}\mod{p^2}$.
\item[--]
For all primes $p\is1\mod4$: $H_p(-1)\is a+bi\mod{p^2}$, where $a,b$ are integers such that
$p=a^2+b^2$, $a\is(-1)^{(p-1)/4}\mod{4}$ and the sign of $b$ is such that $a+bi$ is a $p$-adic unit. 
\end{itemize}
\end{theorem}

The first statement was proven by Eric Mortenson, \cite{Mort03},
the second essentially by L. van Hamme, \cite{VanHamme86}.
Note that in the latter result $i=\sqrt{-1}\in\Z_p$ because $p\is1\mod4$. 
We call such congruences \emph{supercongruences}, a name that was coined in \cite[p293]{StBe85}.
Another example of a supercongruence is the following result, \cite{Kil06}.

\begin{theorem}[Kilbourn, 2006]
For any odd prime $p$ let
$$
F_p(t):=\sum_{k=0}^{p-1}\binom{-1/2}{k}^4t^k.
$$
Then $F_p(1)\is a_p\mod{p^3}$, where $a_p$ is the coefficient of $q^p$ in the 
$q$-expansion of the product
$$
q\prod_{n>0}(1-q^{2n})^4(1-q^{4n})^4.
$$
The latter expansion is the $q$-series of the unique $\Gamma_0(8)$-modular cusp form of weight 4. 
\end{theorem}

Note that this is a statement modulo $p^3$ ! 
This theorem proved the first case of the famous list of 14, formulated as a conjecture
by F.Rodriguez-Villegas, \cite{FRV01}.
Each entry of this list is a congruence modulo $p^3$ of an evaluation at $t=1$
of a $p$-truncation of a ${}_4F_3$ hypergeometric function. This conjecture has inspired many papers
on the subject and was recently proven entirely in the beautiful paper \cite{LoTuYuZu21} by
Long, Tu, Yui and Zudilin.

It should be remarked here that the $a_p$ in the last theorem comes from the local $\zeta$-function
$X^2-a_pX+p^3$ at $p$ of an associated rigid Calabi-Yau threefold.
When $p$ does not divide $a_p$ the polynomial has a $p$-adic zero $\vartheta_p$ which is a unit
root. Note that $a_p\is\vartheta_p\mod{p^3}$,
so in this case $a_p$ and $\vartheta_p$ are interchangeable in the congruence modulo $p^3$.

From about 2009 on, Zhi-Wei Sun presented an amazing list of a 100 or more conjectures for supercongruences
supported by extensive computer computations. See \cite{ZW-Sun19}, an updated version of
an arxiv preprint from 2009 (arXiv:0911.5665), \cite{ZW-Sun11} and many other papers.
More recently, Zhi-Wei Sun's twin brother, Zhi-Hong Sun, joined in and published many more conjectures
and proofs of conjectures. See for example \cite{ZH-Sun11} and \cite{ZH-Sun21}.
The above citations are certainly not complete,
much more can be found by searching through their papers on arxiv.org. Also it is practically impossible
to give a comprehensive overview of these conjectures (proven or not) here.
But I recommend the interested reader to browse through
these papers to get an impression of the many forms of supercongruences that exist. The conjectures and
results of the Sun-brothers has unleashed a flood of papers, each proving
one or more of the conjectures. 
When browsing through the above papers it becomes clear that supercongruences are a persistent
phenomenon and the question arises whether there exists a common mechanism by which they arise.
The present paper is an attempt to describe such a mechanism. We restrict ourselves to congruences
of the form $F_p(t_0)\is\vartheta_p\mod{p^s}$ with $s=2,3$. Here $F_p(t)$ is the $p$-truncation 
(sum of terms of order $<p$) of a
power series solution $F(t)$ of a nice linear differential equation. In Sun's list
one will find that many items are of this type but equally, many others are not.

We propose an approach which we sketch globally, leaving the details and proofs to the 
remaining sections of this paper. 
Let us start with a power series $F(t)\in1+t\Z\pow t$ which satisfies
a linear differential equation of order $\ge2$ with coefficients in 
$\Z[t]$ and such that it is of MUM-type (maximal unipotent monodromy) at $t=0$.
This means that $F(t)$ satisfies a differential equation of the form 
$$
A_n(t)\theta^ny+A_{n-1}(t)\theta^{n-1}y+\cdots+A_1(t)\theta y+A_0(t)y=0,\quad \theta=t\frac{d}{dt},
$$
where $A_i(t)\in\Z[t]$ for all $i$, $A_n(0)\ne0$ and $A_i(0)=0$ for all $i<n$.
One class of examples is given by hypergeometric series such as
$$
{}_2F_1(1/2,1/2;1|16t),\quad {}_4F_3(1/5,2/5,3/5,4/5;1,1,1|5^5t).
$$
Note that we normalized $t$ so that the coefficients of $F$ are integers.
The first series occurs in our starting example,
the second features as one instance of the conjectures of Rodriguez-Villegas.

Another construction is to take a Laurent polynomial $g(\v x)\in\Z[x_1^{\pm1},\ldots,x_n^{\pm1}]$ in the
variables $x_1,\ldots,x_n$, whose gcd of the coefficients is $1$.
Suppose that the Newton polytope $\Delta$ has $\v 0$ as the unique lattice
point in its interior. Then consider
$$
F(t)=\sum_{k\ge0}g_kt^k,\quad g_k=\text{constant term of }g(\v x)^k.
$$
Examples are $g(x,y)=(1+x)(1+y)(1+xy)/x/y$ with 
$$
g_k=\sum_{m=0}^k\binom{k}{m}^3\quad \text{(Franel numbers)}
$$
or $g(x,y)=(1+x)(1+y)(x+y)/x/y$ with
$$
g_k=\sum_{m=0}^k\binom{m}{k}^2\binom{m+k}{k}\quad \text{(Ap\'ery numbers for $\zeta(2)$)}.
$$
Since the differential equation is of MUM-type there exists a unique second solution of the form
$F(t)\log t+G(t)$ with $G(t)\in t\Q\pow t$. Our starting example has
$$
F(t)=\sum_{k\ge0}\binom{2k}{k}^2t^k
$$
which satisfies $\theta^2y=4t(2\theta+1)^2y$ (note that we normalized $t$). One can compute that
$$
G(t)=\sum_{k>0}\binom{2k}{k}^2\left(\sum_{j=k+1}^{2k}\frac{4}{j}\right)t^k.
$$
Define the so-called \emph{canonical coordinate} by
$$
q(t):=\exp\left(\frac{G(t)}{F(t)}\right).
$$
The inverse power series $t(q)$ is called
the \emph{mirror map}. In many cases of interest we have $t(q)\in\Z\pow q$.
Define $t^\sigma(q):=t(q^p)$. This series can be written as a power series in $\Z\pow t$
and we call it the \emph{excellent $p$-Frobenius lift} of $t$. 
Then, in many cases of interest we have the supercongruence 
\be\label{functional-supercongruence}
\frac{F(t)}{F(t^\sigma)}\is F_p(t)\mod{p^s}
\ee
as power series in $t$, for some $s\ge2$ (usually $s=2,3$). Note that modulo $p$ this
congruence reads $F(t)\is F_p(t)F(t^p)\mod{p}$. This congruence is equivalent to the 
statement that the coefficients of $F(t)$ have the so-called \emph{Lucas property}
modulo $p$. This means that if $F(t)=\sum_{k\ge0}g_kt^k$, then for all $k\ge0$ we have
$g_k\is g_{k_0}g_{k_1}\cdots g_{k_m}\mod{p}$, where $k=k_0+k_1p+\cdots+k_mp^m$ is the
base $p$-expansion of $k$. See for example \cite{HeStr22} and \cite{Be22}.

One goal of this paper is to
describe as many $F(t)$ as possible which satisfy the stronger congruence
\eqref{functional-supercongruence}.
\medskip

The other goal is to specialize $t$ in \eqref{functional-supercongruence} to $t_0\in\Z_p$ to get
supercongruences modulo $p^s$ for $F_p(t_0)$. Since $t^\sigma$ is a power series in $t$
that converges only when $|t|_p<1$ we cannot simply set $t=t_0$ in this power series. 

However, the bulk of the examples given by the Sun-brothers concerns $F(t)$ that satisfy a
differential equation of order $n=2$ or $n=3$. In those cases, at least the ones that we
consider, it turns out that the mirror
map $t(q)$ with $q=e^{2\pi i\tau}$ is a modular function with respect to a group $\Gamma$
and $F(t(q))$ is
a modular form of weight one (when $n=2$) or two (when $n=3$) with respect to $\Gamma$
(with possible character on $\Gamma$). Here $\tau$ is the complex coordinate of the
complex upper half plane $\mathcal{H}$. By abuse of notation we shall often denote 
$t(e^{2\pi i\tau})$ by $t(\tau)$. We hope that this context dependent notation
will not cause too much confusion. The idea is now to use this modular parametrization to carry
out the specialization of $t$ to $t_0$.
\smallskip

Let us start with $n=2$ and illustrate this case via our initial
Legendre example. Using classical theory of the hypergeometric function ${}_2F_1(1/2,1/2,1|16t)$
we get that the mirror map equals
$$
t(q)=q - 8 q^2 + 44 q^3 - 192 q^4 + \cdots =q\prod_{n>0}\frac{(1+q^{2n})^{16}}{(1+q^n)^8}.
$$
The series $t(q)$ is connected with Felix Klein's $\lambda$-invariant via the relation
$16t(e^{2\pi i\tau})=\lambda(2\tau)$. It was also known to Klein that
$$
{}_2F_1(1/2,1/2,1|16t(q))=\sum_{x,y\in\Z}q^{x^2+y^2},
$$
which, after setting $q=e^{2\pi i\tau}$, is a modular form of weight one with respect
to $\Gamma_0(4)$. Let us denote this form by
$\theta_S(q)$, where $S$ is meant to refer to the square lattice $\Z^2$. Then we have
$$
\frac{{}_2F_1(1/2,1/2,1|16t(q))}{{}_2F_1(1/2,1/2,1|16t(q^p))}=\frac{\theta_S(q)}{\theta_S(q^p)}
\ \text{or, by abuse of notation, }\frac{\theta_S(\tau)}{\theta_S(p\tau)}.
$$
Thus the quotient $F(t)/F(t^\sigma)$ is related to a quotient of modular forms. 
In fact it is a modular function with respect to $\Gamma\cap\Gamma_0(p)$ and hence an
algebraic function of $t$. We then specialize $t$ to $t_0=t(\tau_0)$, where
$\tau_0\in\mathcal{H}$ is such that $\tau_0=A(p\tau_0)$ for some $A\in\Gamma_0(4)$.
One sees from this that $\tau_0$ is either imaginary quadratic or rational. 
Suppose that $A=\begin{pmatrix}a & b\\c & d\end{pmatrix}$, then
$$
\frac{\theta_S(\tau_0)}{\theta_S(p\tau_0)}=\frac{\theta_S(A(p\tau_0))}{\theta_S(p\tau_0)}
=(-1)^{(d-1)/2}(cp\tau_0+d)
$$
by the modularity of $\theta_S$. Roughly speaking we can now deduce statements like
Theorem \ref{legendre-example} starting from \eqref{functional-supercongruence} with $s=2$.
In this simple minded sketch it turns out that $\tau_0$ depends on the choice of $p$.
In our later proofs we will make some more clever choices.

To make things a bit more precise we describe the general setup. 

\begin{assumption*}[{\bf I}]\label{assumption2}
We make the following assumptions, where $q=e^{2\pi i\tau}$.
\begin{enumerate}
\item $t(q)\in q+q^2\Z\pow t$ is a Hauptmodul for the modular group $\Gamma_0(N)$ for
some integer $N$. Its poles are given by a single $\Gamma_0(N)$-orbit 
denoted by $[\tau_\infty]$, where $\tau_\infty\in\mathcal{H}\cup\Q$.
\item $\Theta(q)\in 1+q\Z\pow q$ is a modular form of weight 1 (with odd quadratic character $\chi$
modulo $N$) with respect to $\Gamma_0(N)$. 
\item The zero-set of $\Theta(q)$ is given by $[\tau_\infty]$.
We assume that the zero-order is $r\le1$.
\item $\Theta(q)$ can be written as an $\eta$-product or $r<1/2$.
\item Let $F\in\Z\pow t$ be such that $\Theta(q)=F(t(q))$.
\end{enumerate}
\end{assumption*}

\begin{remark}\ 

\begin{itemize}
\item[--]
In a number of examples the zero-order $r$ can be strictly less than $1$. This 
happens at elliptic points of $\Gamma_0(N)$.
\item[--]
The number of zeros (counted with multiplicity)
of a weight 1 form in a fundamental domain of $\Gamma_0(N)$
is $|SL(2,\Z):\Gamma_0(N)|/12$. Since $\Theta$ has
a single zero of order $r\le1$ we have $|SL(2,\Z):\Gamma_0(N)|\le12$, hence $N<10$ or $N=11$.
\item[--]
We say that $\Theta(q)$ can be written as $\eta$-product if it can be written
as $\Theta(q)=\prod_{m|N}\eta_m(q)^{e_m}$ where $\eta_m(q)=\eta(q^m)$ and
$\eta(q)$ is the Dedekind $\eta$-function $q^{1/24}\prod_{n\ge1}(1-q^n)$. The numbers
$e_m$ are integers.
\item[--]
In Appendix A we have collected some data on the modular subgroups $\Gamma_0(N)$ for
small $N$ together with our standard choice for the Hauptmodul $h_N$. 
\item[--]
In Appendix B we collected some pairs $t(q),\Theta(q)$ listed as hypergeometric
cases and Zagier {\bf A,B,C,D,E,F}. The coefficients of 
$F(t)$ in these cases form infinite sequences of combinatorial interest. For example, in case {\bf D}
the coefficients of $F(t)$ are the Ap\'ery numbers corresponding to Ap\'ery's irrationality proof
of $\zeta(2)$. Assumption (I) holds for the hypergeometric cases
$F(1/2,1/2;1|16t),F(1/3,2/3,1|27t)$ and Zagier {\bf A,B,C,E,F}. 
Computer experiment suggests that Theorem \ref{thm:functional2} below also holds for the remaining cases.
\end{itemize}
\end{remark}

We have the following general theorem.
\begin{theorem}\label{thm:functional2}
Suppose Assumption (I) holds. Then 
$$
F(t)\is F_p(t)F(t^\sigma)\mod{p^2}
$$
as congruence in $\Z_p\pow t$ for any prime $p$ not dividing $N$.
\end{theorem}

This theorem will be proven in Section \ref{sec:second_order}.

It follows from Assumption (I) that $F$ satisfies a linear differential
equation of order 2 with coefficients in $\Z[t]$.
This is explained in Don Zagier's papers \cite[\S 5]{Zagier09} and 
\cite[\S 5]{Zagier07}. The paper \cite{Zagier07} is devoted to second
order differential equation having solutions in $\Z\pow t$. 
So rather than starting with $F$ as a solution of a linear differential equation
we have now taken $t(q)$ and $\Theta(q)$ as point of departure. 

The next step is to specialize (\ref{functional-supercongruence}). 
As sketched above we do this by setting $t_0=t(\tau_0)$ where $\tau_0\in\mathcal{H}$
is such that $\tau_0$ and $p\tau_0$ are $\Gamma_0(N)$-equivalent. 
It is a classical fact that there exists a polynomial $\Phi_p(x,y)\in\Z[x,y]$ of
degree $p+1$ in $x$, such that $\Phi_p(t(p\tau),t(\tau))=0$. The polynomial 
$\Phi_p(x,y)$ is called the modular polynomial of level $p$. Of course this
modular relation also holds for any specialization, in particular at $\tau_0$.
Note that $t(p\tau_0)=t(\tau_0)=t_0$, hence $\Phi(t_0,t_0)=0$. 
The solutions $t_0$ arise in two different ways, depending on whether $\tau_0$ is rational
or quadratic imaginary. 

When $\tau_0\in\Q$, we have a cusp. The $\Gamma_0(N)$-orbits of a cusp correspond 
to the singularities of the differential equation having local exponents $(0,0)$.
We call the values $t(\tau_0)$ \emph{cusp values} in that case. We have the following
theorem.

\begin{theorem}\label{thm:cuspidal2}
Let notations and assumptions be as in Assumption (I). Let $p$ be a prime
not dividing $N$. Then we have the congruence $F_p(t_0)\is \epsilon_p\mod{p^2}$
for the entries listed in the following table. The cases refer to Appendix B. We use the notation 
$\chi_{-m}(p)=\left(\frac{-m}{p}\right)$ (Kronecker symbol). 
\begin{center}
\begin{tabular}{|l|c|c|c|}
\hline
{\rm case} & $N$ & $t_0$ & $\epsilon_p$\\
\hline
$F(1/2,1/2,1|16t)$ & $4$ & $1/16$ & $\chi_{-4}(p)$\\
$F(1/3,2/3,1|27t)$ & $3$ & $1/27$ & $\chi_{-3}(p)$\\
{\rm Zagier} {\bf A} & $6$ & $1/8,-1$ & $\chi_{-3}(p),\chi_{-3}(p)$\\
{\rm Zagier} {\bf C} & $6$ & $1/9,1$ & $\chi_{-3}(p),1$\\
{\rm Zagier} {\bf F} & $6$ & $1/9,1/8$ & $\chi_{-3}(p),1$\\
{\rm Zagier} {\bf E} & $8$ & $1/8,1/4$ & $\chi_{-4}(p),1$\\
\hline
\end{tabular}
\end{center}
\end{theorem}

\begin{corollary}
The first entry of this table is the first statement in Theorem \ref{legendre-example},
a classical result by Mortensen \cite{Mort08}.
\end{corollary}
For a proof of Theorem \ref{thm:cuspidal2} see Section \ref{sec:second_order}.

The second class of solutions of $\Phi_p(t_0,t_0)=0$ corresponds to $t_0=t(\tau_0)$
where $\tau_0$ is imaginary quadratic. We call them \emph{CM-values}.
It is well-known that they lie in a certain ray classfield extension of $\Q(\tau_0)$,
although that doesn't concern us here. Following a variation on the idea sketched
for the Legendre example we will prove in Section \ref{sec:second_order} the 
following theorem.

\begin{theorem}\label{thm:CMpoints2}
Let notation and assumptions be as in Assumption (I). Let $\alpha$ be an imaginary quadratic
number with positive imaginary part and such that $N|\alpha|^2,N(\alpha+\overline\alpha)\in\Z$.
Let $p$ be a prime not dividing $N$ which splits
in $\Q(\alpha)$. Fix an embedding $\Q(\alpha)\hookrightarrow \Q_p$.
Suppose there exist integers $c,d$ such that
\begin{enumerate}
\item $N|c$
\item $p=|c\alpha+d|^2$ and $c\alpha+d$ is a $p$-adic unit. 
\end{enumerate}
Suppose that $t(\alpha)$ is a $p$-adic integer (possibly in an extension of $\Q_p$). Then
we have the supercongruence $F_p(t(\alpha))\is\chi(d)(c\alpha+d)\mod{p^2}$, where $\chi$ is the
odd character modulo $N$ from Assumption (I).
\end{theorem}

\begin{remark}\ \
\begin{enumerate}
\item 
Note that $\alpha$ has a denominator dividing $N$.
\item We recall again that, by abuse of notation, we write $t(\alpha)$ for $t(e^{2\pi i\alpha})$. 
\item
In the examples in Appendix B we have listed all occurrences of $t(\alpha)\in\Q$. 
\item
Unfortunately we do not see how the
method of this paper could give congruences modulo $p^2$ if $p$ is inert in $\Q(\alpha)$.
\end{enumerate}
\end{remark}

\begin{corollary}[Van Hamme, 1986]
Let $p$ be a prime with $p\is1\mod4$. Write $p=u^2+v^2$ with $v$ even and choose the signs of $u,v$
such that $u+vi$ is a $p$-adic unit. Then
$$\sum_{k=0}^{p-1}\binom{2k}{k}^2\left(\frac{-1}{16}\right)^k\is (-1)^{(u-v-1)/2}(u+vi)\mod{p^2}.$$
\end{corollary}

\begin{proof}
Let $\alpha=(1+i)/2$, then $t(\alpha)=-1/16$. Note that $p=|u+vi|^2=|u-v+2v(1+i)/2|^2$.
The corollary follows by application of Theorem \ref{thm:CMpoints2} with $c=2v$ and $d=u-v$. 
\end{proof}

As another application we consider the case Zagier {\bf F} (twisted with $(-1)^n$)
in Appendix B. There we see that
$t(\tau)$ is a Hauptmodul of $\Gamma_0(6)$ and we have the rational CM-values
$t((3+\sqrt{-3})/6)=-1/12$ and $t((3+\sqrt{-3})/12)=-1/6$.

\begin{corollary}
Let $p$ be a prime with $p\is1\mod3$. Write $p=u^2+3v^2$ and chose the signs of $u,v$ such that $u+v\sqrt{-3}$
is a $p$-adic unit. Then, letting $F(t)$ be the power series solution in the case Zagier {\bf F}
(twisted with $(-1)^n$), we find that 
$$F_p(-1/12)\is F_p(-1/6)\is \chi_{-3}(u)(u+v\sqrt{-3})\mod{p^2}.$$
\end{corollary}

\begin{proof}
Let us take $\alpha=(3+\sqrt{-3})/6$.
Rewrite $p=|u+v\sqrt{-3}|^2=|u-3v+6v\alpha|^2$. We can now apply Theorem \ref{thm:CMpoints2}
with $c=6v$ and $d=u-3v$. Note that $\gcd(c,d)=1$ because $u,v$ cannot have the same parity.

Similarly, when $\alpha=(3+\sqrt{-3})/12$ we can apply Theorem \ref{thm:CMpoints2}
with $c=12v$ and $d=u-3v$.
\end{proof}

This proves Conjecture 3.2(i) in \cite{ZH-Sun21}.

There is a small extension of Theorem \ref{thm:CMpoints2} which yields a new type of congruences.

\begin{theorem}\label{thm:CMpoints2plus}
Consider the case hypergeometric $F(1/2,1/2;1|16t)$ from Appendix B, so $N=4$.
Let assumptions and notations be as in Theorem \ref{thm:CMpoints2} except that we
replace assumption (1) $4|c$ by $c\is2\mod4$.
Then we get $F_p(t(\alpha))\is\pm(c\alpha+d)(16t(\alpha))^{-1/2}\mod{p^2}$. 
\end{theorem}

We did not fight for the determination of the $\pm$sign. Also note the extra
factor $(16t(\alpha))^{-1/2}$ on the right. A proof is given in 
the "Extras" section \ref{sec:extras}. For an application see part two of Corollary
\ref{quadratic}.

Finally we consider some quadratic CM-values in the hypergeometric example $F(1/2,1/2,1|16t)$.
Usually there is an abundance of such examples. In the present case we found 24 CM-values that arise
by starting with the 8 numbers
$$
\frac{(1\pm\sqrt2)^2}{16},\quad\frac{(1\pm\sqrt2)^4}{16},
\quad \frac{2\pm\sqrt3}{4},\quad \frac{1\pm3\sqrt{-7}}{32}
$$ 
and then apply the substitution $x\to \frac{1}{16}-\frac{1}{256x}$ to each of these numbers repeatedly. 
We shall content ourselves by an illustration with $t(\sqrt{-2}/2)=(1-\sqrt{2})^2/16$
and $t((2+\sqrt{-2})/6)=(1+\sqrt2)^2/16$. 

\begin{corollary}\label{quadratic}
Suppose that $p$ is a prime of the form $p=u^2+2v^2$ with $u,v$ integers. Choose the signs
of $u,v$ such that $u+v\sqrt{-2}$ is a p-adic unit. Suppose $v$ is even, i.e. $p\is1\mod8$.
Then
$$\sum_{k=0}^{p-1}\binom{2k}{k}^2\left(\frac{1-\sqrt2}{4}\right)^{2k}\is
\sum_{k=0}^{p-1}\binom{2k}{k}^2\left(\frac{1+\sqrt2}{4}\right)^{2k}
\is (-1)^{(u-1)/2}(u+v\sqrt{-2})\mod{p^2}.$$
Suppose $v$ is odd, i.e. $p\is3\mod8$. Then
$$\sum_{k=0}^{p-1}\binom{2k}{k}^2\left(\frac{1-\sqrt2}{4}\right)^{2k}\is
\pm(u+v\sqrt{-2})(1+\sqrt2)\mod{p^2}.$$
\end{corollary}

\begin{proof}
Recall that $t(q)$ (or $t(\tau)$) is a Hauptmodul for $\Gamma_0(4)$ in this case.
Suppose $v$ is even.
Let $\alpha=\sqrt{-2}/2$. Then $p=|u+v\sqrt{-2}|^2=|u+2v\alpha|^2$.
We can apply Theorem \ref{thm:CMpoints2} with $c=2v,d=u$. 

Let $\alpha=(2+\sqrt{-2})/6$. Then $p=|u+v\sqrt{-2}|^2=|u-2v+6v\alpha|^2$. 
We can apply Theorem \ref{thm:CMpoints2} with $c=6v,d=u$. 

Suppose $v$ is odd. Let $\alpha=\sqrt{-2}/2$. Then
$p=|u+v\sqrt{-2}|^2=|u-2v+6v\alpha|^2$. We can apply Theorem \ref{thm:CMpoints2plus} with $c=2v,d=u$.
Note that $16t(\alpha)=(1-\sqrt{2})^2$, hence $(16t(\alpha))^{-1/2}=\pm(1+\sqrt2)$. 
\end{proof}

When $p\is3\mod{8}$ the argument $t(\alpha)=(1-\sqrt{2})^2/16$ is quadratic over $\Q_p$.
We shall not pursue the subject of algebraic arguments $t_0$ any further
and restrict ourselves to arguments $t_0\in\Z_p$. 
\medskip

Let us proceed with the case $n=3$. There exist many differential equations of order 3 having a
power series solution in $t$ with (almost) integer coefficients. In Appendix C one finds a list
containing hypergeometric examples, the so-called sporadic Ap\'ery-like sequences found by
Almkvist-Zudilin and Cooper and finally two examples that are formed by a Hadamard product
of the coefficients. The most famous one is the equation having the generating series of 
the Ap\'ery numbers for $\zeta(3)$ as solution. We are now in the following setting.

\begin{assumption*}[{\bf II}]
Let $\Gamma$ be an Atkin-Lehner extension of $\Gamma_0(N)$
(see Section \ref{sec:modular-preparation} for the definition of Atkin-Lehner extension).
Assume that
\begin{enumerate}
\item $t(\tau)=q+q^2\Z\pow q$ is a Hauptmodul for $\Gamma$.
\item $\Theta(\tau)=1+q\Z\pow q$ is modular form of weight $2$
with a character that is trivial when restricted to $\Gamma_0(N)$.
\item $\Theta(\tau)$ has a unique zero (modulo the action of $\Gamma$) which is located at 
the pole of $t(\tau)$. We assume it has order $r$ with $r\le1$.
\item 
$\Theta(q)$ can be written as an $\eta$-product or $r<1/3$.
\item Let $F(t)\in\Z\pow t$ be such that $\Theta(\tau)=F(t(\tau))$. 
\end{enumerate}
\end{assumption*}

\begin{theorem}\label{thm:functional3}
Suppose the assumptions in Assumption (II) hold.  
Let $p$ be a prime that does not divide $N$. Then
$$
F(t)\is F_p(t)F(t^\sigma)\mod{p^3}
$$
as power series in $t$. 
\end{theorem}

The proof will be given in Section \ref{sec:third_order}.

\begin{remark}\,
\begin{enumerate}
\item Observe that we have a congruence modulo $p^3$ rather than $p^2$. 
\item The conditions of the theorem are fullfilled for all cases in Appendix C
except hypergeometric $F(1/6,1/2,5/6;1,1|1728t)$, Cooper's $s_7,s_{10},s_{18}$
and the Hadamard products. The case $F(1/6,1/2,5/6;1,1|1728t)$ is covered in
Section \ref{sec:extras}. In the remaining exceptional cases computer experiments indicate
that we have only $F(t)\is F_p(t)F(t^\sigma)\mod{p^2}$.
\end{enumerate}
\end{remark}

Here is the analogue of Theorem \ref{thm:CMpoints2}.

\begin{theorem}\label{thm:CMpoints3}
Let notation and assumptions be as in Assumption (II). Let $\alpha$ be an imaginary quadratic
number with positive imaginary part such that $N|\alpha|^2,N(\alpha+\overline\alpha)\in\Z$.
Let $p$ be a prime not dividing $N$ which splits
in $\Q(\alpha)$. Fix an embedding $\Q(\alpha)\hookrightarrow \Q_p$.
Suppose there exist integers $c,d$ such that
\begin{enumerate}
\item $N|c$.
\item $p=|c\alpha+d|^2$ and $c\alpha+d$ is a $p$-adic unit. 
\end{enumerate}
Suppose that $t(\alpha)$ is a $p$-adic integer (possibly in an extension of $\Q_p$).
Then we have the supercongruence $F_p(t(\alpha))\is(c\alpha+d)^2\mod{p^3}$.
\end{theorem}

We also present a generalization of Theorem \ref{thm:CMpoints3} for cases when
$\Gamma$ contains an Atkin-Lehner involution. 

\begin{theorem}\label{thm:CMpoints3plus}
Suppose Assumption (II) holds. Let $\alpha$ be an imaginary quadratic
number with positive imaginary part. Let $p$ be a prime not dividing $N$ which splits
in $\Q(\alpha)$. Fix an embedding $\Q(\alpha)\hookrightarrow \Q_p$. 

Suppose there exists $w_e\in\Gamma$ with determinant $e$.
Assume that $N|\alpha|^2\in\Z$ and $(N/e)(\alpha+\overline\alpha)\in\Z$. 
Suppose there exist integers $c,d$ such that 
\begin{enumerate}
\item $N|c$, $e|d$.
\item $ep=|c\alpha+d|^2$ and $c\alpha+d$ is a $p$-adic unit. 
\end{enumerate}
Suppose that $t(\alpha)$ is a $p$-adic integer (possibly in an extension of $\Q_p$).
Then we have the supercongruence $F_p(t(\alpha))\is\frac{\chi(w_e)}{e}(c\alpha+d)^2\mod{p^3}$.
\end{theorem}

These theorems will be proven in Section \ref{sec:third_order}.
Many of the steps are similar to the proof of Theorem \ref{thm:CMpoints2}.

\begin{remark}\ \
\begin{itemize}
\item[--]
Theorem \ref{thm:CMpoints3plus} with $e=1$ is precisely Theorem \ref{thm:CMpoints3}.
\item[--]
In Appendix C one also finds listings of the rational CM-values that we found.
In some cases there can be many of them. For hypergeometric $F(1/4,1/2,3/4;1,1|256t)$ we found
$15$ rational values $t(\alpha)$. On the other hand, in the case of Ap\'ery numbers we found only
$t((2+\sqrt{-2})/6)=1$ and $t((3+\sqrt{-3})/6)=-1$. 
\item[--]
If $p$ does not split in $\Q(\alpha)$ we do not know if our modular form approach is
possible.
\end{itemize}
\end{remark}

Here are some corollaries of Theorem \ref{thm:CMpoints3} and \ref{thm:CMpoints3plus}.
The first corollary concerns the 
hypergeometric example $F(1/2,1/2,1/2;1,1|64t)$.
In Appendix C we see that the group is $\Gamma_0(4)+\left<4\right>$ and we have
$t(i/2)=1/64$ and $t(1/2+i)=-1/512$. These CM-values (multiplied by $64$) occur in the paper
\cite[Thm 3,Thm 29]{LoRa16} of Long
and Ramakrishna. In slightly altered form they prove the following.
\begin{theorem}[Long, Ramakrishna (2016)]
Let $p$ be a prime of the form $p=u^2+v^2$ with $v$ even. Then we have
the congruence
$$
\sum_{k=0}^{p-1}\binom{2k}{k}^3t_0^k\is \begin{cases}(u+vi)^2\mod{p^3}&\text{when $t_0=1/64$}\\
\left(\frac{2}{p}\right)(u+vi)^2\mod{p^3}&\text{when $t_0=-1/512$}\end{cases}
$$
\end{theorem}

This is not a literal citation because Long and Ramakrishna use 
$-\Gamma_p(1/4)^4$ instead of $(u+vi)^2$ and we normalized $t_0$. 
In \cite{LoRa16} the authors also deal with the primes $p\is3\mod4$,
i.e. the primes that do not split in the CM-field.
Unfortunately we do not see how our modular approach would give results for such primes. 
For the case $t_0=1/64$ we reprove the above theorem with the modular method.

\begin{proof}
Let us take $t_0=t(i/2)=1/64$ and $\alpha=i/2$. Write $p=|u+vi|^2=|u+2v(i/2)|^2$. Apply Theorem
\ref{thm:CMpoints3} with $c=2v,d=u$. 
\end{proof}
For $t_0=-1/512$ we need both Theorem \ref{thm:CMpoints3} and Theorem \ref{thm:CMpoints3plus}.
\begin{proof}
Let us take $t_0=t(1/2+i)=-1/512$ and $\alpha=1/2+i$. Write $p=|u+vi|^2=|u-v/2+v(i+1/2)|^2$.

Suppose $p\is1\mod8$. Then $v\is0\mod4$ and $u-v/2$ is odd.
Let $c=v$ and $d=u-v/2$. Then the conditions of Theorem \ref{thm:CMpoints3} are satisfied
and our congruence follows.

Suppose $p\is5\mod8$. Then $v\is2\mod4$ and both $v$ and $u-v/2$ are even.
Let $c=2v$ and $d=2u-v$. Hence $4p=|c\alpha+d|^2$. The modular group
contains $w_4:\tau\to-1/4\tau$ and we have $\Theta(-1/4\tau)=-4\tau^2\Theta(\tau)$.
Hence $\chi(w_4)=-1=\left(\frac{2}{p}\right)$. The conditions of Theorem \ref{thm:CMpoints3plus} with $e=4$
are satisfied and so our congruence follows.
\end{proof}

The second corollary concerns the hypergeometric example $F(1/4,1/2,3/4;1,1|256t)$.
In Appendix C we see that the modular group is $\Gamma_0(2)+\left<2\right>$ and we have 
$t((1+\sqrt{-5})/2)=-1/1024$. 

\begin{corollary}\label{cor:p3example1}
Let $p$ be a prime of the form $1,9\mod{20}$. Write it in the form $p=u^2+5v^2, u,v\in\Z$.
Choose the sign of $u,v$ such that $u+v\sqrt{-5}$ is a $p$-adic unit. Then
$$
\sum_{k=0}^{p-1}\frac{(4k)!}{k!^4}\ \left(\frac{-1}{1024}\right)^k
\is(u+v\sqrt{-5})^2\mod{p^3}.
$$
Let $p$ be a prime of the form $3,7\mod{20}$. Write it in the form $p=2u^2+2uv+3v^2, u,v\in\Z$.
Choose $u,v$ such that $2u+v(1+\sqrt{-5})$ is a $p$-adic unit. Then
$$
\sum_{k=0}^{p-1}\frac{(4k)!}{k!^4}\ \left(\frac{-1}{1024}\right)^k
\is -\frac12(2u+v(1+\sqrt{-5}))^2\mod{p^3}.
$$
\end{corollary}
\begin{proof}
When $p\is1,9\mod{20}$ rewrite $u+v\sqrt{-5}$ as $u-2v+2v\alpha$ with $\alpha=(1+\sqrt{-5})/2$. 
Apply Theorem \ref{thm:CMpoints3} with $c=2v,d=u-2v$ to find our statement. 

When $p\is3,7\mod{20}$ rewrite $p=2u^2+2uv+3v^2$ as $p=2|u+v\alpha|^2$ with $\alpha=(1+\sqrt{-5})/2$. 
Apply Theorem \ref{thm:CMpoints3plus} with $e=2,c=2v,d=2u$ to find our statement. 
\end{proof}

Similarly one can write down supercongruences for the 14 other rational CM-values listed in
the case $F(1/4,1/2,3/4;1,1|256t)$ in Appendix C.

Our third corollary concerns the Ap\'ery numbers $A_n=\sum_{k=0}^n\binom{n}{k}^2\binom{n+k}{n}^2$.
Let $F(t)=\sum_{k\ge0}A_kt^k$. Then we find in Appendix C that the group is $\Gamma_0(6)+\left<6\right>$
and we have $t((2+\sqrt{-2})/6)=1$.

\begin{corollary}\label{cor:p3example2}
Let $p$ be a prime which can be written in the form $p=u^2+2v^2, u,v\in\Z$. Choose the sign of $u,v$
such that $u+v\sqrt{-2}$ is a $p$-adic unit. Then
$$
\sum_{k=0}^{p-1}A_k\is(u+v\sqrt{-2})^2\mod{p^3}.
$$
\end{corollary}
\begin{proof}
Rewrite $u+v\sqrt{-2}$ as $u-2v+6v\alpha$ with $\alpha=(2+\sqrt{-2})/6$. Apply Theorem \ref{thm:CMpoints3}
with $c=6v,d=u-2v$ to find our statement. 
\end{proof}

The mod $p^2$-version of this congruence was proven in \cite{SunWang19}. 
In Z.W.Sun's list \cite{ZW-Sun19} one finds many examples of the above type in the form of conjectures,
many of them modulo $p^2$.

\medskip

In the literature on supercongruences a prominent role is played by the so-called
Van Hamme-congruences.
They are named after Lucien van Hamme (1939 - 2014), a Belgian mathematician who did some pioneering work
in this subject, see \cite{VanHamme86} and \cite{VanHamme97}. For an updated account of
the Van Hamme-congruences see Swisher's paper \cite{Swisher15}.
The corresponding differential equations have (mostly) order $3$ and in the congruence also the derivative of
$F_p(t)$ occurs. A prototype is given by the following.

\begin{corollary}\label{VanHammeConjecture}
Let $p$ be an odd prime. Then
$$
\sum_{k=0}^{p-1}(4k+1)\binom{2k}{k}^3\left(\frac{-1}{64}\right)^k\is (-1)^{(p-1)/2}p\mod{p^3}.
$$
\end{corollary}
This was a conjecture by Van Hamme, proven independently by Mortenson \cite{Mort08},
Long \cite{Long11} and Zudilin \cite{Zud09}. We will prove it below as corollary of the following theorem.

\begin{theorem}\label{thm:VanHamme2}
Suppose Assumption (II) holds. 
Let $\alpha$ be an imaginary quadratic number with positive imaginary part
and denominator dividing $N$.
Suppose $\alpha$ is not a zero of $t'$, the derivative of $t$.
Let $p$ be a prime not dividing $N$ which splits
in $\Q(\alpha)$. Fix an embedding $\Q(\alpha)\hookrightarrow\Q_p$.

Suppose there exists $w_e\in\Gamma$ with determinant $e$.
Assume that $N|\alpha|^2\in\Z$ and $(N/e)(\alpha+\overline\alpha)\in\Z$. 
Suppose there exist integers $c,d$ such that 
\begin{enumerate}
\item $N|c$, $e|d$.
\item $ep=|c\alpha+d|^2$ and $c\alpha+d$ is a $p$-adic unit. 
\end{enumerate}
Define
\[
\delta(\alpha)=\frac{t(\alpha)}{t'(\alpha)}\left(\frac{2}{\alpha-\overline\alpha}+
\frac{\Theta'(\alpha)}{\Theta(\alpha)}\right),
\]
where $'$ denotes differentiation with respect to $\tau$. 
Then $\delta(\alpha)\in \Q(\alpha,t(\alpha))$. Suppose that $\delta(\alpha)$
is a $p$-adic unit and $t(\alpha)$ a $p$-adic integer. Then
\begin{equation}\label{general-vanhamme}
F_p(t(\alpha))-\frac{1}{\delta(\alpha)}(\theta F_p)(t(\alpha))\is \chi(w_e)p\mod{p^3}.
\end{equation}
Explicitly, if $F(t)=\sum_{k\ge0}g_kt^k$ we have
$$
\sum_{k=0}^{p-1}\left(1-\frac{k}{\delta(\alpha)}\right)g_kt(\alpha)^k\is\chi(w_e)p\mod{p^3}.
$$
\end{theorem}

The condition that $t'(\alpha)\ne0$ (i.e. $\alpha$ is not an elliptic point) is important,
otherwise $\delta(\alpha)$ cannot be computed.
Vanishing of $t'(\alpha)$ happens for example in the case
$F(1/2,1/2,1/2;1,1|64t)$ and $\alpha=i/2, t(i/2)=1/64$. Nevertheless we have the following result
by Z.W.Sun and C.Wang, \cite[(1.12)]{SunWang23}.
\begin{theorem}
For all primes $p$ with $p\is1\mod4$ we have
$$
\sum_{k=0}^{p-1}(1+4k)\binom{2k}{k}^3\frac{1}{64^k}\is0\mod{p^2}.
$$
\end{theorem}
Perhaps surprisingly the congruence does not hold when $p\is3\mod4$. 

\begin{proof}[Proof of Corollary \ref{VanHammeConjecture}]
This corresponds to the case hypergeometric $F(1/2,1/2,1/2;1,1|64t)$ in Appendix C.
There we see that $t(\alpha)=-1/64$ with $\alpha=(1+\sqrt{-2})/2$. Moreover,
$\delta(\alpha)=-1/4$.

Suppose $p\is1,3\mod8$. Then $p$ can be written in the form $p=u^2+2v^2$ for some
integers $u,v$. Rewrite this as $p=|u-v+2v\alpha|^2$. 

Suppose $p\is1\mod8$, then $v$ is even and $u$ odd. Application of
Theorem \ref{thm:VanHamme2} with $e=1,c=2v,d=u-2v$ yields the congruence. 

Suppose $p\is3\mod8$, then $u,v$ are both odd and we can write $4p=|2(u-v)+4v\alpha|^2$. 
Application of Theorem \ref{thm:VanHamme2} with $e=4,c=4v,d=2(u-v)$ yields the congruence.
\end{proof}

As usual we do not know if a modular approach works for the cases where $p$ is
inert in the CM-field $\Q(\sqrt{-2})$, i.e. $p\is5,7\mod8$.

As further application of Theorem \ref{thm:VanHamme2}
we displey the Van Hamme congruences corresponding to
Corollaries \ref{cor:p3example1} and \ref{cor:p3example2}. 

\begin{corollary}
Let notations and assumptions be as in Corollary \ref{cor:p3example1}. We can compute that 
$\delta((1+\sqrt{-5})/2)=-3/20$. Then, for any prime $p\is1,3,7,9\mod{20}$ and $p>3$ we get
$$
\sum_{k=0}^{p-1}\left(1+\frac{20}{3}k\right)\frac{(4k)!}{k!^4}\left(\frac{-1}{1024}\right)^k\is
(-1)^{(p-1)/2}p\mod{p^3}.
$$
Let notations and assumptions be as in Corollary \ref{cor:p3example2}. Then $\delta((2+\sqrt{-2})/6)=-1/2$ and we get
$$
\sum_{k=0}^{p-1}(1+2k)A_k\is p\mod{p^3}.
$$
\end{corollary}
The first statement was first proven by Zudilin \cite[(9)]{Zud09}, including the primes $p$
that do not split in $\Q(\sqrt{-5})$.
\smallskip

It has been observed many times that the Van Hamme-congruences are very closely linked to the famous Ramanujan series
for $1/\pi$. Having the modular background at hand the connection is quite straightforward.

\begin{theorem}\label{thm:Ramanujan}
Choose a differential equation from Appendix C. Suppose that the mirror map $t\in q+O(q^2)$ is a Hauptmodul
for the modular group $\Gamma$. Let $\Theta=F(t(\tau))$ and suppose that $\Theta(\tau)^r$ is a modular form
with respect to $\Gamma$ of weight $2r$ with trivial character for some integer $r\ge1$.
Then, with the notation from Theorem \ref{thm:VanHamme2}, we have
$$
\sum_{k\ge0}\left(1-\frac{k}{\delta(\tau)}\right)g_kt(\tau)^k=t(\tau)
\frac{2\pi i\Theta(\tau)}{t'(\tau)}\frac{1}{\delta(\tau)}\frac{1}{\pi i(\tau-\overline\tau)}.
$$
In particular $(2\pi i\Theta t/t')^r\in\Q(t)$ and $\delta$ is a real analytic function
invariant under $\Gamma$. 
\end{theorem} 

If we choose the argument $\tau=\alpha$ in a quadratic imaginary number field, then $t(\alpha)$ and
$\delta(\alpha)$ will have algebraic values and we obtain fancy looking series expansions for $1/\pi$.

The proof of Theorem \ref{thm:Ramanujan} is quite direct and given in Section \ref{sec:third_order}.
An important application is the famous formula for $1/\pi$ given by the Chudnovsky brothers.
Take the case hypergeometric $F(1/6,1/2,5/6;1,1|1728t)$. Then we see from Appendix C that $t(\tau)=1/J(\tau)$,
where $J$ is the classical $J$-invariant and we have $\Theta(\tau)^2=E_4(\tau)$, the classical 
Eisenstein series of weight $4$. We can also check that $2\pi i\Theta t/t'=1/\sqrt{1-t}$.
Let us take $\alpha=(1+\sqrt{-163})/2$. Then 
$$t(\alpha)=-\frac{1}{640320^3}\quad\text{and}\quad \delta(\alpha)=-\frac{545140134}{13591409},$$
which yields the Chudnovsky formula. We present it in a slightly unconventional way. 
$$
\sum_{k\ge0}\left(1+\frac{545140134}{13591409}k\right)
\binom{6k}{3k}\binom{3k}{k}\binom{2k}{k}\frac{(-1)^k}{640320^{3k}}=\frac{426880\sqrt{10005}}{13591409\pi}.
$$

There are endless further applications of Theorem \ref{thm:Ramanujan}
in the form of Ramanujan-Sato series and beyond, see \cite{ChanCooper12}.

Naturally, Theorem \ref{thm:VanHamme2} gives analogues in the form of supercongruences in many cases
and many of them are conjectures in Sun's list.
Unfortunately, the example $F(1/6,1/2,5/6;1,1|1728t)$ is just outside the reach of Theorem
\ref{thm:functional3} and its consequences, the reason being that $\Theta$ itself is not a modular form.
This is not a problem for Theorem \ref{thm:Ramanujan}, but for our supercongruences it is.
Fortunately, as we remarked above, $\Theta^2=E_4$, which is a modular form.
In Section \ref{sec:extras} we shall tweak our methods to obtain the following theorem.

\begin{theorem}\label{thm:RamanujanVanHamme}
Let $D$ be a positive integer which is $0$ or $3$ modulo $4$. Write $\omega_D=\sqrt{-D}/2$ if $D\is0\mod4$
and $\omega_D=(1+\sqrt{-D})/2$ if $D\is3\mod4$. Let $\wp$ be an algebraic integer in $\Q(\sqrt{-D})$
such that $\wp\overline\wp=p$, where $p$ is a prime in $\Z$. Suppose, after embedding $\sqrt{-D}$ in
$\Q_p$, that $\wp$ is a $p$-adic unit. Then,
$$
\sum_{k=0}^{p-1}\binom{6k}{3k}\binom{3k}{k}\binom{2k}{k}\frac{1}{J(\omega_D)^k}\is \pm\wp^2\mod{p^3}
$$
and
$$
\sum_{k=0}^{p-1}\left(1-\frac{k}{\delta(\omega_D)}\right)
\binom{6k}{3k}\binom{3k}{k}\binom{2k}{k}\frac{1}{J(\omega_D)^k}\is \pm p\mod{p^3}.
$$
\end{theorem}
We were unable to give an expression for the $\pm$ sign. Note also that in the second congruence
the prime $\wp$ does not occur. But we are able to prove it only for primes $p$ that split
in $\Q(\sqrt{-D})$.
\smallskip

When $n\ge4$ the examples of interest do not have a modular background and the nature of the
mirror map is much more mysterious.
Recall the \emph{excellent Frobenius lift} which we defined at the beginning of this introduction.
It was first introduced by B. Dwork in his famous paper
\cite[\S 7]{dwork69}, known as '$p$-Adic Cycles'. 
We speak of a Frobenius lift because $t^\sigma(q)=t(q^p)\is t(q)^p\is t^p\mod{p}$.
The substitution $A(t)\mapsto A(t^\sigma)$ gives an endomorphism
of the ring $\Z\pow t$ and modulo $p$ it comes down to 
$A(t)\mapsto A(t^\sigma)\is A(t^p)\is A(t)^p\mod{p}$.

The excellent Frobenius lift has a modular interpretation in the
case of the Legendre family of elliptic curves (as utilized by Dwork)
and more generally the equations of order 2 and 3 we dealt with until now.
Lacking such an interpretation in general we can still make a few observations.
For example, it turns out that $t^\sigma$ is a $p$-adic limit of rational functions in $t$. 
Dwork states in \cite{dwork69} that this is a consequence of 'Deligne's theorem'
(without giving a reference) applied to the Legendre case. 
For the case of constant term sequences with symmetric $g(\v x)$ it is proven in 
\cite[Thm 7.15]{BeVlIII}. Consequently this allows us to compute the specialization of $t^\sigma$ at 
some $p$-adic point $t_0$ by specializing the approximating rational functions at $t_0$
and then take the limit. The image is denoted by $t_0^\sigma$. A condition is that $t_0$ is not a
pole modulo $p$ of the approximating rational functions which, in many cases, comes down to the condition
$F_p(t_0)\not\is0\mod{p}$. The second evaluation problem is to compute $F(t_0)/F(t_0^\sigma)$
which, again, we cannot do by direct substitution into the power series. However it was Dwork's
insight that $F(t)/F(t^\sigma)$ can be approximated $p$-adically by the functions
$F_{p^r}(t)/F_{p^{r-1}}(t^\sigma)$, see Theorem \ref{thm:dworkcongruence} for the case of
constant term sequences $F(t)$.
By specializing these rational functions at $t=t_0$ and taking the limit one can compute
$F(t_0)/F(t_0^\sigma)$ in principle. Another insight of Dwork was that the limit
$F(t_0)/F(t_0^\sigma)$ equals
the unit root $\vartheta_p$ of a certain $\zeta$-function if $t_0^\sigma=t_0$. It must be
said that the latter two statements are not contained in \cite{dwork69} in full generality.
We refer to Theorem \ref{thm:dworkcongruence} and the Appendix of \cite{BeVlI} instead.

So we have argued that one may expect supercongruences at arguments $t_0$ that are fixed under
the excellent $p$-Frobenius lift. Preferably one would like to have rational or
algebraic numbers $t_0$ that are fixed for infinitely many primes $p$. In our modular setting
we looked for $\tau_0$ such that $\tau_0$ and $p\tau_0$ are $\Gamma$-equivalent, and hence
$t(p\tau_0)=t(\tau_0)$. But lacking such a background the nature of these fixed points is much less
clear.
We conjecture that the singularities of the differential equation are among
them, as long as they lie in $\Z_p$. The point $1$ is a singularity of the hypergeometric equations
and the (former) conjectures by Rodriguez-Villegas are examples of supercongruences that correspond
to them. In \cite{AESZ10} we find an extensive list of fourth order
equations of Calabi-Yau type, precisely the kind we are considering. It would be interesting to see
if Rodriguez-Villegas's list extends to at least some of them. 
Recently there has been some interest in special points called attractor points.
These are special values of $t$ that are 
related to the existence of black hole solutions of certain supergravity theories in theoretical
physics. Regrettably we lack the expertise to go into this subject more deeply.
What is relevant for us is that there are arithmetic methods to search for such
points. For example, consider the family of Calabi-Yau threefolds given by $1-tg(\v x)=0$ with
$$
g=(1+x_1+x_2+x_3+x_4)(1+1/x_1+1/x_2+1/x_3+1/x_4).
$$
This is called the \emph{Hulek-Verrill} family. The corresponding constant terms read
$$
{\rm hv}_k:=\sum_{m_1+m_2+m_3+m_4+m_5=k}\left(\frac{k!}{m_1!m_2!m_3!m_4!m_5!}\right)^2. 
$$
Candelas-De la Ossa-Elmi-Van Straten \cite{COES19} discovered experimentally that at
$t_0=-1/7$ the local zeta-function of
$1-t_0g(\v x)=0$ factors as $(X^2-a_pX+p^3)(X^2-pb_px+p^3)$ with $a_p,b_p\in\Z$
for many $p\ne7$. Supposedly this identifies $-1/7$ as an attractor point (of rank 2).
Computer experiment suggests the following conjecture.
\begin{conjecture}\label{hulek-verrill-conj}
For all primes $p\ne7$ we have 
$$
\sum_{k=0}^{p-1}{\rm hv}_k(-1/7)^k\is a_p\mod{p^2},
$$
where $a_p$ is coefficient of $q^p$ in the weight 4 cusp form of level
14 starting with $q-2q^2+8q^3+\cdots$
\end{conjecture}
The differential equation associated to this case is \#34 in the list of \cite{AESZ10}. It
reads $Ly=0$ with
\begin{eqnarray*}
L&=&\theta^4-t(35\theta^4+70\theta^3+63\theta^2+28\theta+5)\\
&&+t^2(\theta+1)^2(259\theta^2+518\theta+285)-225t^3(\theta+1)^2(\theta+2)^2.
\end{eqnarray*}
The finite singularities are $t_0=1,1/9,1/25$. For these arguments there is an analogue 
for Conjecture \ref{hulek-verrill-conj} with the level 6 form
$$
q\prod_{n\ge1}(1-q^n)^2(1-q^{2n})^2(1-q^{3n})^2(1-q^{6n})^2
$$
when $t_0=1,1/9$ and the level 30 form $q-2q^2+3q^3+4q^4+5q^5-6q^6+\cdots$ when $t_0=1/25$.
Another example of an attractor point is $t_0=-1/8$ for the hypergeometric function
${}_4F_3(1/3,1/3,2/3,2/3;1,1,1|t)$. This point was found in 
\cite[Table 1]{BEKK22}. A supercongruence modulo $p^3$ at $t_0=1$ occurs in the list of
Rodriguez-Villegas, but experiment suggests the following.
\begin{conjecture}
For all primes $p\ne3$ we have 
$$
\sum_{k=0}^{p-1}\binom{-1/3}{k}^2\binom{-2/3}{k}^2(-1/8)^k\is a_p\mod{p^3},
$$
where $a_p$ is coefficient of $q^p$ in the weight 4 cusp form of level 54 starting with
$q+2q^2+4q^4+3q^5+\cdots$
\end{conjecture}
Note that we seem to have a congruence modulo $p^3$. 

Finally we like to mention a much more general congruence which is a super version
of the following.
\begin{theorem}[Dwork congruence]\label{thm:dworkcongruence}
Suppose $F(t)$ is the generating function of a constant term sequence. Then, for any prime
$p$ and any $r\ge1$ we have
\be\label{mellit-vlasenko}
\frac{F(t)}{F(t^\sigma)}\is\frac{F_{p^r}(t)}{F_{p^{r-1}}(t^\sigma)}\mod{p^r}
\ee
as power series in $t$.
\end{theorem}
A proof occurs after \cite[Thm 3.2]{BeVlII}. However, this proof uses the Frobenius
lift $t\to t^p$ instead of $t\to t^\sigma$.
Fortunately one can easily adapt it to the case of any Frobenius lift, including the excellent
Frobenius lift. A precursor of \cite[Thm 3.2]{BeVlII} is in \cite{MeVl16}. 

Surprisingly, it seems that whenever \eqref{functional-supercongruence} holds, we also have
the modulo $p^{2r}$-congruence
$$
\frac{F(t)}{F(t^\sigma)}\is\frac{F_{p^r}(t)}{F_{p^{r-1}}(t^\sigma)}\mod{p^{2r}}
$$
for all $r\ge1$. The paper \cite{BeVlIII} was an attempt to prove this in at least one
case, but unfortunately we failed. Among the few known proven examples is an indication of an approach
given in the Guo-Zudilin paper \cite{GuoZud21} and the references therein. Remarkably
enough one gets Van Hamme type results modulo $p^{3r}$ via an approach based on so-called
$q$-supercongruences.
\medskip

{\bf Acknowledgement}: Many thanks to the referees whose valuable comments have helped significantly
to improve the manuscript.

\section{Central binomial coefficients, an example when $n=1$}
Before plunging into proofs announced in the introduction we like to point
out a very simple instance of our program when $n=1$. Consider the series
$$
F(t)=\sum_{k\ge0}\binom{2k}{k}t^k=\frac{1}{\sqrt{1-4t}}.
$$
Clearly it satisfies the differential equation $(1-4t)\theta y=2ty$, but we like
to consider the extended equation
$$
\theta((1-4t)\theta-2t)y=0.
$$
So, in a way we are back again in the case $n=2$, albeit a bit degenerate.
A basis of solutions of this differential equation is given by
$$
\frac{1}{\sqrt{1-4t}},\quad \frac{1}{\sqrt{1-4t}}\times
\log\left(\frac{2t}{1-2t+\sqrt{1-4t}}\right).
$$
The $q$-coordinate is simply
$$
q=\frac{2t}{1-2t+\sqrt{1-4t}}
$$
and the mirror map
$$
t=\frac{q}{(1+q)^2}.
$$
Hence $t^\sigma=\frac{q^p}{(1+q^p)^2}$, from which we see that $1/t^\sigma$ is
a polynomial of degree $p$ in $1/t$.
Furthermore, $F(t)=\frac{1+q}{1-q}$ and hence, for any odd prime $p$,
$$
\frac{F(t)}{F(t^\sigma)}=\frac{(1+q)(1-q^p)}{(1-q)(1+q^p)}. 
$$
Let $pG(t)=(1+q)^p-1-q^p$ and notice that by expansion in a $p$-adic geometrical series
modulo $p^2$,
$$
\frac{1}{1+q^p}=\frac{1}{(1+q)^p-pG(t)}\is \frac{1}{(1+q)^p}+\frac{pG(t)}{(1+q)^{2p}}\mod{p^2}.
$$
Hence
$$
\frac{(1+q)(1-q^p)}{(1-q)(1+q^p)}\is\frac{1-q^p}{1-q}\left(\frac{1}{(1+q)^{p-1}}
+\frac{pG(t)}{(1+q)^{2p-1}}\right)\mod{p^2}.
$$
One easily verifies that the rational function on the right is invariant under $q\mapsto 1/q$
and that the only pole is in $-1$ with pole order $<2p$. Hence the right hand side is a 
polynomial $A_p(t)$ in $t$ of degree $<p$. Hence $F(t)/F(t^\sigma)\is A_p(t)\mod{p^2}$.
Since $F(t^\sigma)$ has the form $1+O(t^p)$ we conclude that $A_p(t)\is F_p(t)\mod{p^2}$. 
So congruence \eqref{functional-supercongruence} holds in this case and in particular we
see that
$$
F_p(t)\is \frac{(1+q)(1-q^p)}{(1-q)(1+q^p)}\mod{p^2}\quad \text{with }t=\frac{q}{(1+q)^2}.
$$
This is a rational function version of a result by Z.W.Sun, \cite{ZW-Sun10} with $a=1$
and $m=q+2+1/q$. See also \cite{ZW-Sun21}.

We can specialize the identity at will. Some examples,

\begin{itemize}
\item[--] $q=e^{2\pi i/3}$. Then $t=1$ and we find
$$
\sum_{k=0}^{p-1}\binom{2k}{k}\is\epsilon\mod{p^2}
\quad \text{where}\ \epsilon=\pm1\ \text{and}\ p\is\epsilon\mod{3}.
$$
See \cite[(1.9)]{SunTau11}
\item[--] $q=i$. Then $t=1/2$ and we find
$$
\sum_{k=0}^{p-1}\binom{2k}{k}\frac{1}{2^k}\is\epsilon\mod{p^2}
\quad \text{where}\ \epsilon=\pm1\ \text{and}\ p\is\epsilon\mod{4}
$$
\item[--] $q=e^{\pi i/3}$. Then $t=1/3$ and we find
$$
\sum_{k=0}^{p-1}\binom{2k}{k}\frac{1}{3^k}\is\epsilon\mod{p^2}
\quad \text{where}\ \epsilon=\pm1\ \text{and}\ p\is\epsilon\mod{6}
$$
\item[--] $q=1$. Then $t=1/4$ and we find
$$ 
\sum_{k=0}^{p-1}\binom{2k}{k}\frac{1}{4^k}\is p\mod{p^2}
$$
\item[--] $q=e^{4\pi i/5}$. Then $t=\frac{3+\sqrt{5}}{2}$ and we find
$$ 
\sum_{k=0}^{p-1}\binom{2k}{k}\left(\frac{3+\sqrt{5}}{2}\right)^k\is
\begin{cases}\epsilon\mod{p^2} & \text{when}\ p\is\epsilon\mod{5}\\
\epsilon(2-\sqrt{5})\mod{p^2} & \text{when}\ p\is2\epsilon\mod{5}
\end{cases}
$$ 
\end{itemize}
In the literature we find an abundance of supercongruences involving many 
different truncations of the series $F(t)$. Two more or less random references
are \cite{MaoTau20} and \cite{SunTau11}.

\section{Modular preparations}\label{sec:modular-preparation}
In the two sections following this one we deal with differential equations of order $n=2,3$.
These are the cases where almost always the mirror map $t(q)$ is a modular function.
We remind the reader that we often abuse our notation by using the notations $t(q)$ and
$t(\tau)$ interchangebly, where $\tau$ is the parameter of the complex upper half plane. 
By the latter we actually mean $t(e^{2\pi i\tau})$, but
we trust that this context dependent notation will not present too many difficulties.
We start with a few simple lemmas

\begin{lemma}\label{lemma:elementary}
Let $h(q)\in\Z\pow q$. Then, for any prime $p$ we have
$$
\prod_{k=0}^{p-1}(1+h(q^{1/p}\zeta_p^k)x)\is 1+h(q)x^p\mod{p},
$$
where $\zeta_p=e^{2\pi i/p}$. In particular, when $0<s<p$, the $s$-th elementary function in
$h(q^{1/p}\zeta^k),k=0,1,\ldots,p-1$ is a power series in $\Z\pow q$ whose coefficients are divisible by $p$.
The $p$-th elementary symmetric function equals $h(q)$ modulo $p$. 
\end{lemma}

\begin{proof}
We prove our equality modulo $\zeta_p-1$. Since the coefficients of both sides of the equality are in $\Z$
the congruence modulo $p$ follows immediately. We see that the product equals
$$
\prod_{k=0}^{p-1}(1+h(q^{1/p}\zeta_p^k)x)\is \prod_{k=0}^{p-1}(1+h(q^{1/p})x)\is (1+h(q^{1/p})x)^p
\is 1+h(q)x^p\mod{\zeta_p-1}. 
$$
Since $p$ is unramified in $\Q(\zeta_p)$ this congruence also holds modulo $p$.
\end{proof}

\begin{lemma}\label{lemma:etaproduct}
Let $m$ be a positive integer and $p$ a prime that does not divide $m$.
Define $\eta_m(\tau):=\eta(m\tau)$ where $\eta(\tau)=
q^{1/24}\prod_{n\ge1}(1-q^n)$.
Then
$$
\eta_m(p\tau)\prod_{k=0}^{p-1}\eta_m((\tau+k)/p)=\exp(2\pi im(p-1)/48)\times\eta_m(\tau)^{p+1}.
$$
\end{lemma}

\begin{remark}
A fortiori this lemma holds for any product of $\eta_m$
with different $m$, all of them not divisible by $p$. We call a product of the
form $\prod_{m|N}\eta_m^{e_m}$ with $e_m\in\Z$ an \emph{$\eta$-product}.
\end{remark}

\begin{proof}
We prove our statement for $m=1$, the case $m>1$ being completely similar.
Observe that 
$$\exp(2\pi ip\tau/24)\prod_{k=0}^{p-1}\exp(2\pi i(\tau+k)/(24p))=
\exp(2\pi i(p-1)/48)\exp(2\pi i(p+1)\tau/24),$$
which accounts for the factors coming from $q^{1/24}$. Next, for any $n\ge1$,
$$\prod_{k=0}^{p-1}(1-\zeta_p^{nk}q^{n/p})=\begin{cases}1-q^n &\text{if $p$ does not divide $n$}
\\(1-q^{1/p})^p &\text{if $p$ divides $n$} \end{cases}
$$
Hence
$$
\prod_{n\ge1}\prod_{k=0}^{p-1}(1-\zeta^{nk}q^{n/p})=\prod_{n\ge1}(1-q^n)\times
\prod_{n\is0\mod{p}}\frac{(1-q^{n/p})^p}{1-q^n}.
$$
Replace $n$ by $np$ in the second product to get
$$
\prod_{n\ge1}(1-q^n)\times \prod_{n\is0\mod{p}}\frac{(1-q^{n/p})^p}{1-q^n}
=\prod_{n\ge1}\frac{(1-q^n)^{p+1}}{1-q^{np}}.
$$
This finishes the proof of our lemma. 
\end{proof}

\begin{lemma}\label{lemma:cuspidalzero}
Consider an $\eta$-product $E(\tau)=\prod_{m|N}\eta_m(\tau)^{e_m}$.
Let $r,s$ be relatively prime integers. Then $r/s$ is a zero of $E(\tau)$
if and only if
$$
\sum_{m|N}e_m\frac{\gcd(s,m)^2}{m}>0.
$$
\end{lemma}

\begin{proof}
Substitute $\tau=r/s+i\epsilon$ in $E(\tau)$ and let $\epsilon\downarrow0$. 
For each factor we have $\eta_m(\tau)=\eta(mr/s+im\epsilon)$. Let $m'=m/\gcd(s,m)$
and $s'=s/\gcd(s,m)$. Choose $u,v\in\Z$ such that $m'rv-us'=1$. Then by modularity 
of $\eta$ we get
$$
\eta\left(\frac{v\tau-u}{-s'\tau+m'r}\right)\sim(-s'\tau+m'r)^{1/2}\eta(\tau),
$$
where $\sim$ denotes equality up to a non-zero factor independent of $\epsilon$.
Set $\tau=mr/s+im\epsilon=m'r/s'+mi\epsilon$. Then after some rearranging,
$$
\eta(mr/s+im\epsilon)\sim\epsilon^{-1/2}\eta\left(\frac{i}{m(s')^2\epsilon}-\frac{v}{s'}\right)
\sim \epsilon^{-1/2}\exp\left(-\frac{\gcd(s,m)^2}{24ms^2\epsilon}\right).
$$
Hence $E(\tau)$ behaves like 
$$
E(r/s+i\epsilon)\sim\epsilon^{-\sum_me_m/2}\exp\left(-\frac{1}{24s^2\epsilon}\sum_{m\in N}e_m\frac{\gcd(s,m)^2}{m}\right)
$$
and our lemma follows. 
\end{proof}

The following lemma is a consequence of the standard theory of modular forms.

\begin{lemma}\label{lemma:cusplimit}
Let $e|N$ be such that $\gcd(e,N/e)=1$ or $2$. Then any point $r/s$ with $\gcd(r,s)=1,r\ne0,s>0$
and $\gcd(s,N)=e$ is $\Gamma_0(N)$-equivalent to $1/e$. 

Let $f$ be a modular form of weight $1$ with respect to $\Gamma_0(N)$ and character $\chi$.
Suppose that $f$ has no zero in $1/e$, i.e. $\lim_{\epsilon\downarrow0}\epsilon f(1/e+i\epsilon)=c_e\ne0$.
Let $d\in\Z$ be such that $d\is r^{-1}\mod{e}$ and $d\is s/e\mod{N/e}$.
Then $\lim_{\epsilon\downarrow0}\epsilon f(r/s+i\epsilon)=\chi(d)ec_e/s$ for any $r/s$ as above.
\end{lemma}

\begin{proof}
To prove the first statement choose $a,b,c,d\in\Z$ such that
$$
d\is r^{-1}\mod{s},\quad d\is s/e\mod{N/e},\quad b=(rd-1)/s,\quad c=s-de,\quad 
a=r-be.
$$
Note that $\gcd(s,N/e)=\gcd(e,N/e)$. If this gcd is $1$ then the congruence equations
for $d$ can be satisfied. If the gcd is $2$ then $s$ and $N/e$ are both even. So
$r,s/e$ are both odd and the congruence equations can also be satisfied in that case. 
One easily verifies that
$$
A:=\begin{pmatrix}a & b\\ c & d\end{pmatrix}\in\Gamma_0(N)\quad
\text{and}\quad a+be=r,c+de=s.  
$$  
Set $\tau=1/e+i\epsilon$ in $f(A\tau)=\chi(d)(c\tau+d)f(\tau)$ to obtain
$$
f\left(\frac{r}{s}+i\epsilon\frac{1}{(c/e+d)^2}+O(\epsilon^2)\right)=\chi(d)(c/e+d+ci\epsilon)
f(1/e+i\epsilon).
$$
Multiply by $\epsilon$ and take the limit as $\epsilon\downarrow0$. We find
$$
\lim_{\epsilon\downarrow0}\epsilon f\left(\frac{r}{s}+i\epsilon\frac{e^2}{s^2}\right)
=\chi(d)\frac{s}{e}c_e.
$$
Replace $\epsilon$ by $\epsilon s^2/e^2$ to find the second assertion of our lemma. 
\end{proof}

Finally we introduce the \emph{Atkin-Lehner extensions} $\Gamma$ of $\Gamma_0(N)$ and
classify the orbits of $\Theta(p\tau),t(p\tau)$ under the action of $\Gamma$ on $\tau$.

Let $N$ be a positive integer and consider $\Gamma_0(N)$.
Next, let  $e$  be a \textit{Hall divisor} of $ N $; that is, $ e | N $ and $\gcd(e, N/e) = 1.$ We define the set of \textit{Atkin--Lehner involutions} $ W_e $ by matrices of the form
\begin{equation}\label{AL}
w_e:=\begin{pmatrix}
a_e & b_e \\
c_e N & d_e 
\end{pmatrix} \in \mathrm{GL}_2^{+}(\mathbb{Z}),
\end{equation}
where $ a_e, b_e, c_e, d_e $ are integers so that  $e\mid a_{e}$, $e\mid d_{e}$, and 
$ \det(w_e) = e $. In particular $W_1=\Gamma_0(N)$. 
One easily verifies that
$w_e$ normalizes $\Gamma_0(N)$ and the product of any two elements from $W_e$ is 
$e$ times an element in $\Gamma_0(N)$.
Consequently the union $\Gamma_0(N)\cup W_e$ forms a subgroup of
$\mathrm{PGL}_2(\Q)^+$.
 
When $ e = N $, we call the matrix
$$
w_N =
\begin{pmatrix}
0 & -1 \\
N & \phantom{-}0
\end{pmatrix}.
$$
the \textit{Fricke involution} for $\Gamma_0(N)$. 
This matrix satisfies $ \det(w_N) = N $ and $ w_N^2 = -N{\rm Id}_2$.
In general $w_e$ need not contain an exact involution, i.e. a matrix with zero trace. 

We note that the Hall divisors of $N$ form a multiplicative set with the rule that
multiplication of two Hall divisors $e,e'$ equals $e'':=ee'/\gcd(e,e')^2$. One easily
checks that the product of an element in $W_e$ with an element in $W_{e'}$ gives an
element in $W_{e''}$ (up to scalars). 

For any multiplicative subset $S$ of the set of Hall divisors of $N$, we define the group
\[
\Gamma_0(N) + S := \cup_{e\in S}W_e.
\]
This group is called an \textit{Atkin-Lehner extension} of $ \Gamma_0(N) $ by the involutions
corresponding to $S$. We call $\Gamma_0(N)$ itself also a (trivial) Atkin-Lehner extension.
In this paper we shall specify the set $S$ by its generators. For example, when $N=6$, the
set of Hall divisors $\{1,2,3,6\}$ is denoted by $\left<2,3\right>$.

When $S$ is the set of all Hall divisors we obtain the \textit{full Atkin--Lehner extension}
denoted by $\Gamma_0(N)+$. 

From the above properties it also follows that $\Gamma_0(N)+S$ is a normal extension of
$\Gamma_0(N)$ with quotient group isomorphic to $(\Z/2\Z)^s$, where $s$ is the number of
generators of $S$.

Let $k$ be an integer and $\gamma=\begin{pmatrix}a & b \\ c & d \end{pmatrix}
\in \mathrm{GL}_2^+(\mathbb{R})$. Let $f(\tau)$ be any function from 
$\mathcal{H}$ to $\C$.
The \textit{weight $k$ slash operator} is defined by
$$
f|_{k}\gamma(\tau):=\mathrm{det}(\gamma)^{\frac{k}{2}}(c\tau+d)^{-k}f(\gamma\cdot\tau).
$$
Furthermore, we say $f(\tau)$ is a \textit{modular form of weight $k$ with multiplier $\chi$ for $\Gamma,$} a discrete subgroup of $\mathrm{GL}_2^{+}({\mathbb{R}}),$ if for any $\gamma\in \Gamma,$
 
$$f|_{k}\gamma(\tau)=\chi(\gamma)f(\tau),$$ 

where $\chi:\Gamma\rightarrow\mathbb{C}$ satisfies $|\chi(\gamma)|=1.$ To ease notation, when the weight~$k$ of $f$ is clear from the context, we write $f|\gamma$ for $f|_{k}\gamma$.

We now consider the action of $\Gamma$ on $t(p\tau)$ and $\Theta(p\tau)$.
To do this in an elegant way we consider the set of matrices
$$
M(N,p):=\left\{\left.\begin{pmatrix}a & b\\ cN & d\end{pmatrix}\right|a,b,c,d\in\Z,
ad-Nbc=p\right\}.
$$
Clearly there is a left action of $\Gamma_0(N)$ on $M(N,p)$ and we denote its quotient
set by $\Gamma_0(N)\b M(N,p)$. It consists of $p+1$ elements for which we have a
standard set of representatives given by
$$
\begin{pmatrix}p & 0\\0 & 1\end{pmatrix},\quad \begin{pmatrix}1 & b\\0 & p\end{pmatrix}
\quad(b=0,1,\ldots,p-1).
$$
Clearly the right action of $\Gamma_0(N)$ permutes these cosets and one easily verifies
that this action is transitive. More generally, any $w\in\Gamma_0(N)+$ acts on 
$\Gamma_0(N)\b M(N,p)$ via $\gamma\mapsto w^{-1}\gamma w$. 

In the next lemma we need to extend $\chi$ (see Lemma \ref{gamma-orbit1})
to $M(N,p)$. This can be done
in a natural way by assigning to $\gamma\in M(N,p)$ the value $\chi(d)$ where
$d$ is the lower right hand entry of $\gamma$. We assume here that the character
$\chi$ is already given on $\Gamma_0(N)$ in that way. One naturally has
$\chi(\sigma\gamma)=\chi(\sigma)\chi(\gamma)$  for all $\sigma\in\Gamma_0(N)$
and $\Gamma\in M(N,p)$. 

\begin{lemma}\label{gamma-orbit1}
Let $p$ be a prime not dividing $N$ and $H$ a form of weight $1$ with respect to
$\Gamma_0(N)$ with odd character $\chi$.
Then $\Gamma_0(N)$ permutes the elements of the following set,
$$
\left\{\chi(\gamma)^{-1}\frac{H|\gamma}{H}\right\}_{[\gamma]\in\Gamma_0(N)\b M(N,p)}.
$$
One easily verifies that the elements of this set are independent of the choice of the $\gamma$'s
in their equivalence class $[\gamma]\in\Gamma_0(N)\b M(N,p)$. 
\end{lemma}

\begin{proof}
Let $\sigma\in\Gamma_0(N)$. Then for any $\gamma\in M(N,p)$ there exist $\gamma'\in M(N,p)$ and
$\xi\in\Gamma_0(N)$ such that $\xi\gamma'=\gamma\sigma$. We easily see that
$$
\chi(\gamma)^{-1}\frac{H|\gamma\sigma}{H|\sigma}=
\chi(\gamma)^{-1}\frac{H|\xi\gamma'}{H|\sigma}=
\chi(\gamma)^{-1}\chi(\xi)\chi(\sigma)^{-1}\frac{H|\gamma'}{H}=
\chi(\gamma')^{-1}\frac{H|\gamma'}{H}.
$$
Here we used $H|\sigma=\chi(\sigma)H$ for any $\sigma\in\Gamma_0(N)$
and $\chi(\xi)\chi(\gamma')=\chi(\gamma)\chi(\sigma)$. 
\end{proof}

\begin{lemma}\label{gamma-orbit2}
Let $\Gamma$ be an Atkin-Lehner extension of $\Gamma_0(N)$ and $p$ a prime 
not dividing $N$. Let $H$ be a modular form of weight $2$
with respect to $\Gamma$ with a character $\chi$ which is trivial 
when restricted to $\Gamma_0(N)$. Then $\Gamma$ permutes the elements of the set
$$
\left\{\frac{H|\gamma}{H}\right\}_{[\gamma]\in\Gamma_0(N)\b M(N,p)}.
$$
One easily verifies that the elements of this set are independent of the choice of
$\gamma$ in its equivalence class $[\gamma]\in\Gamma_0(N)\b M(N,p)$. 
\end{lemma}

\begin{proof}
Let $w\in\Gamma$ and define $\gamma'=w^{-1}\gamma w$ for any $\gamma\in M(N,p)$.
Then 
$$
\frac{H|\gamma w}{H|w}=
\frac{H|w\gamma'}{\chi(w)H}=
\frac{\chi(w)H|\gamma'}{\chi(w)H}=
\frac{H|\gamma'}{H}.
$$
Here we used $H|w=\chi(w)H$.
\end{proof}

\section{Second order equations}\label{sec:second_order}
This section will be mainly devoted to the proof Theorem \ref{thm:functional2}
and its consequences. Let notations and assumptions be as in Assumption (I).
Let $p$ be a prime not dividing $N$ and consider the polynomial
\be\label{Tp-definition}
T_p(x)=\left(1-\frac{\Theta(p\tau)}{\Theta(\tau)}x\right)\prod_{k=0}^{p-1}
\left(p-\chi(p)^{-1}\frac{\Theta((\tau+k)/p)}{\Theta(\tau)}x\right).
\ee
It has $\Theta(\tau)/\Theta(p\tau)$ as zero. We can rewrite it as
$$
T_p(x)=p^p\prod_{[\gamma]\in\Gamma_0(N)\b M(N,p)}
\left(1-\chi(\gamma)^{-1}\frac{\Theta|\gamma}{\sqrt{p}\Theta}x\right). 
$$
By Lemma \ref{gamma-orbit1} we see that the coefficients of this polynomial are
modular functions with respect to $\Gamma_0(N)$. Moreover the poles of these
coefficients are located at the zeros of $\Theta(\tau)$, hence the coefficients
are polynomials in $t$. The degree of the coefficient of $x^m$ is $\le mr$,
where $r$ is the zero order of $\Theta$.
Thus we have constructed a polynomial $T_p(x,y)$ such that
$T_p(\Theta(\tau)/\Theta(p\tau),t(\tau))=0$.
\begin{example}
We display $T_p$ for $p=5,7$ in the hypergeometric case $F(1/2,1/2;1|16t)$.
\[
T_5=5^5 - 2\cdot5^5 x + 7\cdot5^4 x^2 - 12\cdot5^3 x^3 + 11\cdot5^2 x^4-( 26 - 2^{12} y + 
 2^{16} y^2 )x^5 + x^6
\]
and
\begin{eqnarray*}
T_7&=&7^7 - 4\cdot7^6 x^2 - 16\cdot7^5(1 - 2 y) x^3 - 30\cdot7^4 x^4 - 32\cdot7^3( 1 - 2^5 y) x^5 \\
&&- 4\cdot7^2(5 + 3\cdot2^{10} y - 3\cdot2^{14} y^2) x^6 - 
16(1 - 2^5 y) (3 - 2^{11} y + 2^{15} y^2) x^7 - x^8
\end{eqnarray*}
\end{example}

\begin{proposition}\label{theta-mirror2}
Let $p$ be a prime not dividing $N$ and $T_p(x,y)$ the polynomial we just constructed.
Denote the coefficient of $x^m$ by $T_{p,m}(y)$. Then
\begin{enumerate}
\item[(i)] $T_p(x,y)\in\Z[x,y]$.
\item[(ii)] $T_p(x,y)\is T_{p,p}(y)x^p+T_{p,p+1}(y)x^{p+1}\mod{p^2}$.
\item[(iii)] $T_{p,p+1}(y)\is\chi(p)\mod{p}$.
\item[(iv)] If $\Theta$ has an $\eta$-product representation then $T_{p,p+1}(y)=\chi(p)$
and the degree of $T_{p,p}$ is at most $p-1$. 
\end{enumerate}
\end{proposition}

\begin{proof}
We first determine the $q$-expansion of \eqref{Tp-definition}. Since $\Theta(q)\in1+q\Z\pow q$
we can deduce that $T_p(x)\in\Z\pow q[x]$. Hence $T_p(x)\in\Z\pow t[[x]]$. Since
the coefficients of $T_p(x)$ are polynomials in $t$ we see that (i) follows.

To prove (ii) we observe that by Lemma \ref{lemma:elementary}
$$
\sum_{k=0}^{p-1}\frac{\Theta(\tau)}{\Theta((\tau+k)/p)}=
\sum_{k=0}^{p-1}\frac{\Theta(q)}{\Theta(\zeta_p^kq^{1/p})}\is0\mod{p}.
$$
Hence,
\begin{eqnarray*}
T_p(x)&=&\left(x-\frac{\Theta(q}{\Theta(q^p)}\right)\prod_{k=0}^{p-1}
\left(x-\chi(p)p\frac{\Theta(q)}{\Theta(\zeta_p^kq^{1/p})}\right)T_{p,p+1}(t)\\
&\is&T_{p,p+1}(t)\left(x-\frac{\Theta(q)}{\Theta(q^p)}\right)
\left(x^p-\chi(p)px^{p-1}\sum_{b=0}^{p-1}\frac{\Theta(q)}{\Theta(\zeta_p^bq^{1/p})}
\right)\mod{p^2}\\
&\is&T_{p,p+1}(t)\left(x-\frac{\Theta(q}{\Theta(q^p)}\right)x^p\mod{p^2}.
\end{eqnarray*}
This proves part (ii). To prove (iii) we consider the $q$-expansion of
$$
T_{p,p+1}(t)=\frac{\Theta(q)}{\Theta(q^p)}\prod_{b=0}^{p-1}\frac{\chi(p)\Theta(q)}
{\Theta(\zeta_p^bq^{1/p})}
$$
modulo $p$. By Lemma \ref{lemma:elementary} we find 
$\prod_{b=0}^{p-1}\Theta(\zeta_p^b q^{1/p})\is\Theta(q)\mod{p}$ and
$\Theta(q^p)\is\Theta(q)^p\mod{p}$.
Using this we find that $T_{p,p+1}\is\chi(p)\mod{p}$ as $q$-expansions, hence $t$-expansions.

Finally, the first part of (iv) follows from Lemma \ref{lemma:etaproduct}.
For the second part we observe that
the zero of $\Theta$ must be a cusp, say $a\in\Q$. Since $\gcd(p,N)=1$ Lemma \ref{lemma:cuspidalzero}
implies that all functions $\Theta((\tau+k)/p),k=0,1,\ldots,p$ vanish at $a$. Hence the degree of 
$T_{p,p}$ is stricly less than $pr$, hence $\le p-1$. 
\end{proof}

\begin{proof}[Proof of Theorem \ref{thm:functional2}]
Consider the polynomial $T_p$ found in Proposition \ref{theta-mirror2}. In part (ii)
we found that
$$
T_{p,p}\left(\frac{\Theta(q)}{\Theta(q^p)}\right)^p+
T_{p,p+1}\left(\frac{\Theta(q)}{\Theta(q^p)}\right)^{p+1}
\is0\mod{p^2}.
$$
Hence 
$$
T_{p,p}+T_{p,p+1}\frac{\Theta(q)}{\Theta(q^p)}\is0\mod{p^2}.
$$
Rewritten as power series equality in $t$ we get
$$
T_{p,p}(t)+T_{p,p+1}(t)\frac{F(t)}{F(t^\sigma)}\is0\mod{p^2}.
$$
Suppose $\Theta$ can be written as $\eta$-product. Then, by Proposition
\ref{theta-mirror2}(iv), $T_{p,p+1}=\chi(p)$ and $T_{p,p}$ has degree $\le p-1$.
So modulo $p^2$ the quotient $F(t)/F(t^\sigma)$ equals a polynomial in $t$ of degree $\le p-1$. 
Since $F(t^\sigma)\is 1\mod{t^p}$ it follows that this polynomial must be $F_p(t)$. 

Let us now suppose that $\Theta(q)$ is not an $\eta$-product and assume $r<1/2$.
Then, by Proposition \ref{theta-mirror2}(iii), $T_{p,p+1}(t)=\chi(p)+pB(t)$ for some
polynomial $B(t)\in\Z[t]$ of degree $\le(p+1)r$. Moreover, $T_{p,p}(t)$ has degree $\le pr$.
Multiply by $\chi(p)-pB(t)$ on both sides to obtain
$$
(\chi(p)-pB(t))T_{p,p}(t)+\frac{F(t)}{F(t^\sigma)}\is0\mod{p^2}.
$$
Hence modulo $p^2$ the quotient $F(t)/F(t^\sigma)$ equals a polynomial in $t$ of degree 
$\le \floor{pr}+\floor{(p+1)r}$, which is $\le p-1$ because $r<1/2$.
We finish our proof as before. 
\end{proof}

\begin{corollary}\label{specialcongruence2}
Assume that Assumption (I) holds.
Let $p$ be a prime not dividing $N$ and $T_p(x,y)$ the polynomial constructed above.
Let $\alpha$ be a quadratic imaginary number with positive imaginary part.
Suppose that $t(\alpha)$ is a $p$-adic integer and suppose $T_p(\vartheta,t(\alpha))=0$
has a $p$-adic unit $\vartheta_p$ as solution. Then
$$
F_p(t(\alpha))\is\vartheta_p\mod{p^2}.
$$
\end{corollary}

\begin{proof}
Using the notation from the proof of Proposition \ref{theta-mirror2} we get
$$
T_{p,p+1}(t(\alpha))\vartheta_p^{p+1}+T_{p,p}(t(\alpha))\vartheta_p^p\is0\mod{p^2}.
$$
After division by $\vartheta_p^p$ we get 
$T_{p,p+1}(t(\alpha))\vartheta_p+T_{p,p}(t(\alpha))\is0\mod{p^2}$. From the proof
of Theorem \ref{thm:functional2} we derive that 
$T_{p,p+1}(t(\alpha))F_p(t(\alpha))+T_{p,p}(t(\alpha))\is0\mod{p^2}$. Since
$T_{p,p+1}(t(\alpha))\is\chi(p)\mod{p}$ is a $p$-adic unit, our corollary follows.
\end{proof}

We are now ready to prove Theorems \ref{thm:cuspidal2} and \ref{thm:CMpoints2}.

\begin{proof}[Proof of Theorem \ref{thm:cuspidal2}]
We apply Corollary \ref{specialcongruence2} in the case when $\tau_0\in\Q$ is a cusp
whose denominator divides $N$. It turns out that $T_p(x,t(\tau_0))=(x-\epsilon_p)(x-p\epsilon'_p)^p$
for some $\epsilon_p,\epsilon'_p\in\{\pm1\}$. But we will not prove this. Instead we concentrate
on finding $\epsilon_p$.

We summarize the data from Appendix B for the cases at hand.
\begin{center}
\begin{tabular}{|l|c|c|c|}
\hline
{\rm case} & Level $N$ & $\chi$ & $t_0$ \\
\hline
$F(1/2,1/2,1|t)$ & $4$ & $\chi_{-4}$ & $t(1)=1/16$\\
$F(1/3,2/3,1|t)$ & $3$ & $\chi_{-3}$ & $t(1)=1/27$\\
{\rm Zagier} {\bf A} & $6$ & $\chi_{-3}$ & $t(1)=1/8,t(1/2)=-1$\\
{\rm Zagier} {\bf C} & $6$ & $\chi_{-3}$ & $t(1)=1/9,t(1/3)=1$\\
{\rm Zagier} {\bf E} & $8$ & $\chi_{-4}$ & $t(1)=1/8,t(1/4)=1/4$\\
{\rm Zagier} {\bf F} & $6$ & $\chi_{-3}$ & $t(1/2)=1/9,t(1/3)=1/8$\\
\hline
\end{tabular}
\end{center}
Notice that the arguments $\tau_0$ have the form $1/e$, where $e|N$ and
$\gcd(e,N/e)\in\{1,2\}$. Suppose that $p\is1\mod{e}$. We compute
$\chi(p)p\Theta(\tau)/\Theta((\tau+k)/p)$ with $\tau=1/e+i\epsilon$
and $k=(p-1)/e$ and let $\epsilon\downarrow0$. Note that $(\tau_0+k)/p=1/e+i\epsilon/p$. 
We get
$$
\lim_{\epsilon\downarrow0}\chi(p)p\frac{\Theta(1/e+i\epsilon)}{\Theta(1/e+i\epsilon/p)}=
\chi(p).
$$
When $e=1,2$ we conclude that we get the value $\chi(p)$. When $e=3,4$ (in case $N=6,8$)
we are given that $p\is1\mod{e}$ and $p$ odd, hence $\chi(p)=1$. 

Suppose now that $p\is-1\mod{e}$ and $e=3,4$.
We perform a similar computation as in the previous case,
but now with $k=-(p+1)/e$. Note that $(\tau_0+k)/p=-1/e+i\epsilon/p$. We get
$$
\lim_{\epsilon\downarrow0}\chi(p)p\frac{\Theta(1/e+i\epsilon)}{\Theta(-1/e+i\epsilon/p)}
=\lim_{\epsilon\downarrow0}\chi(p)\frac{\Theta(1/e+i\epsilon)}{\Theta(-1/e+i\epsilon)}.
$$
Now use Lemma \ref{lemma:cusplimit} with $r/s=-1/e$. Then we compute that $d\is-1\mod{e}$
and $d\is1\mod{N/e}$. But $N/e=2$ in all cases at hand, so $d\is-1\mod{N}$. 
Lemma \ref{lemma:cusplimit} then implies that we get the value $\chi(p)\chi(-1)=\chi(-p)=1$. 
\end{proof}

\begin{proof}[Proof of Theorem \ref{thm:CMpoints2}] 
Let $c'=-c$ and $d'=d+c(\alpha+\overline\alpha)$. Note that $c'\alpha+d'=c\overline\alpha+d$
and $p=(c')^2|\alpha|^2+dd'$. Suppose $q$ is a prime dividing $c'$ and $d'$. Then
we see that $q|p$, hence $q=p$. Since $\gcd(p,N)=1$ we find that $p$ also divides
$d$ and $c'|\alpha|^2$. Hence $p^2|p$, which is a contradiction.

We conclude that $\gcd(c',d')=1$. 
Choose $a,b\in\Z$ such that $ad'-bc'=1$. We now apply Corollary \ref{specialcongruence2} with 
$\tau_0=\alpha$. Notice that
$$
\frac{a\alpha+b}{c'\alpha+d'}=\frac{1}{p}(a\alpha+b)(c'\overline\alpha+d')=
\frac{1}{p}(\alpha+ac'|\alpha|^2+bc'(\alpha+\overline\alpha)+bd').
$$
Let $k=ac'|\alpha|^2+bc'(\alpha+\overline\alpha)+bd'$ and compute 
$\chi(p)p\Theta(\alpha)/\Theta((\alpha+k)/p)$. 
By modularity we have 
$$
\Theta\left(\frac{\alpha+k}{p}
\right)=\Theta\left(\frac{a\alpha+b}{c'\alpha+d'}\right)=\chi(d')(c'\alpha+d')\Theta(\alpha)
=\chi(d')(c\overline\alpha+d)\Theta(\alpha).
$$
Note also that $p=c\cdot c|\alpha|^2+dd'$, hence $p\is dd'\mod{N}$ and $\chi(p)=\chi(d)\chi(d')$.
So we find that
$$
\chi(p)p\frac{\Theta(\alpha)}{\Theta((\alpha+k)/p)}=\chi(d)(c\alpha+d),
$$
which is a $p$-adic unit. The desired supercongruence now follows. 
\end{proof}

\section{Third order differential equations}\label{sec:third_order}
Much of this section runs parallel to Section \ref{sec:second_order}
and we shall mainly indicate the differences. 
Let notations and assumptions be as in Assumption (II).
Let $p$ be a prime not dividing $N$ and consider the polynomial
\be\label{Tp-definition2}
T_p(x)=\left(1-\frac{\Theta(p\tau)}{\Theta(\tau)}x\right)\prod_{k=0}^{p-1}
\left(p^2-\frac{\Theta((\tau+k)/p)}{\Theta(\tau)}x\right).
\ee
It has $\Theta(\tau)/\Theta(p\tau)$ as zero. We can rewrite it as
$$
T_p(x)=p^{2p}\prod_{[\gamma]\in\Gamma_0(N)\b M(N,p)}
\left(1-\frac{\Theta|\gamma}{p\Theta}x\right). 
$$
By Lemma \ref{gamma-orbit2} we see that the coefficients of this polynomial are
modular functions with respect to $\Gamma$. Moreover the poles of these
coefficients are located at the zeros of $\Theta(\tau)$, hence the coefficients
are polynomials in $t$. The degree of the coefficient of $x^m$ is $\le mr$,
where $r$ is the zero order of $\Theta$.
Thus we have constructed a polynomial $T_p(x,y)$ such that
$T_p(\Theta(\tau)/\Theta(p\tau),t(\tau))=0$.

\begin{proposition}\label{theta-mirror3}
Let $p$ be a prime not dividing $N$ and $T_p(x,y)$ the polynomial we just constructed.
Denote the coefficient of $x^m$ by $T_{p,m}(y)$. Then
\begin{enumerate}
\item[(i)] $T_p(x,y)\in\Z[x,y]$.
\item[(ii)] $T_p(x,y)\is T_{p,p}(y)x^p+T_{p,p+1}(y)x^{p+1}\mod{p^3}$.
\item[(iii)] $T_{p,p+1}(y)\is1\mod{p}$.
\item[(iv)] If $\Theta$ has an $\eta$-product representation then $T_{p,p+1}(y)=1$
and the degree of $T_{p,p}$ is at most $p-1$. 
\end{enumerate}
\end{proposition}

The proof is very similar to the proof of Proposition \ref{theta-mirror2}.

\begin{proof}[Proof of Theorem \ref{thm:functional3}]
Consider the polynomial $T_p$ found in Proposition \ref{theta-mirror3}.
Following the ame steps as in in the proof Theorem \ref{thm:functional2} we find that
$$
T_{p,p}(t)+T_{p,p+1}\frac{F(t)}{F(t^\sigma)}\is0\mod{p^3}.
$$
Suppose $\Theta$ can be written as $\eta$-product. Then, by Proposition
\ref{theta-mirror3}(iv), $T_{p,p+1}=1$ and $T_{p,p}$ has degree $\le p-1$.
So modulo $p^3$ the quotient $F(t)/F(t^\sigma)$ equals a polynomial in $t$,
which must be $F_p(t)$. 

Let us now suppose that $\Theta(q)$ is not an $\eta$-product and assume $r<1/3$.
Then, by Proposition \ref{theta-mirror3}(iii), $T_{p,p+1}(t)=1+pB(t)$ for some
polynomial $B(t)\in\Z[t]$ of degree $\le(p+1)r$. Moreover, $T_{p,p}(t)$ has degree $\le pr$.
Multiply by $1-pB(t)+p^2B(t)^2$ on both sides to obtain
$$
(1-pB(t)+p^2B(t)^2)T_{p,p}(t)+\frac{F(t)}{F(t^\sigma)}\is0\mod{p^3}.
$$
Hence modulo $p^3$ the quotient $F(t)/F(t^\sigma)$ equals a polynomial in $t$ of degree 
$\le \floor{pr}+2\floor{(p+1)r}$, which is $\le p-1$ because $r<1/3$.
We finish our proof as before. 
\end{proof}

\begin{remark}
We conjecture that Theorem \ref{thm:functional3} also holds when $r=1/3$.
Theorem \ref{thm:functional3extend} gives a proof modulo some ad-hoc congruences. 
\end{remark}

\begin{corollary}\label{specialcongruence3}
Assume that Assumption (II) holds.
Let $p$ be a prime not dividing $N$ and $T_p(x,y)$ the polynomial constructed above.
Let $\alpha$ be a quadratic imaginary number with positive imaginary part.
Suppose that $t(\alpha)$ is a $p$-adic integer and suppose $T_p(\vartheta,t(\alpha))=0$
has a $p$-adic unit $\vartheta_p$ as solution. Then
$$
F_p(t(\alpha))\is\vartheta_p\mod{p^3}.
$$
\end{corollary}

The proof is analogous to the proof of Corollary \ref{specialcongruence2}.

\begin{proof}[Proof of Theorems \ref{thm:CMpoints3} and \ref{thm:CMpoints3plus}]
Note that Theorem \ref{thm:CMpoints3plus} implies Theorem \ref{thm:CMpoints3} by
setting $e=1$. So we prove the latter theorem.

Let $c'=-c$ and $d'=d+c(\alpha+\overline\alpha)$. Note that $c'\alpha+d'=c\overline\alpha+d$
and $ep=(c')^2|\alpha|^2+dd'$. Divide the latter by $e$, $p=(c'/e)c'|\alpha|^2+(d/e)d'$.
Suppose $q$ is a prime dividing both $d'$ and $c'/e$. Then $q|p$, hence $q=p$. Since
$\gcd(p,N)=1$ we find that $p$ also divides $d/e$ and $c'|\alpha|^2$. Hence $p^2|p$,
which is clearly a contradiction.

So $\gcd(c'/e,d')=1$ and we can find integers $a',b$ such that $a'd'-b(c'/e)=1$. Multiply by $e$
and set $a=ea'$. We get $ad'-bc'=e$ and $e$ divides $a,d'$. So $a,b,c',d'$ are the entries of
a matrix in $W_e$, the Atkin-Lehner involutions with determinant $e$. 

Notice that
$$
\frac{a\alpha+b}{c'\alpha+d'}=\frac{1}{p}(a\alpha+b)(c'\overline\alpha+d')=
\frac{1}{ep}(e\alpha+ac'|\alpha|^2+bc'(\alpha+\overline\alpha)+bd')=\frac{\alpha+k}{p} with
$$
$k=(a/e)c'|\alpha|^2+b(c'/e)(\alpha+\overline\alpha)+bd'/e$. Note that $k\in\Z$.
Now we apply  Corollary \ref{specialcongruence3} and compute $p^2\Theta(\alpha)/\Theta((\alpha+k)/p)$. 
By modularity we have 
$$
\Theta\left(\frac{\alpha+k}{p}\right)=
\Theta\left(\frac{a\alpha+b}{c'\alpha+d'}\right)=\chi(w_e)e(c'\alpha+d')^2\Theta(\alpha)=
\chi(w_e)e(c\overline\alpha+d)^2\Theta(\alpha).
$$
So we find that
$$
p^2\frac{\Theta(\alpha)}{\Theta((\alpha+k)/p)}=\frac{\chi(w_e)}{e}(c\alpha+d)^2,
$$
which is the $p$-adic unit $\vartheta_p$. The desired supercongruence now follows.
\end{proof}

\begin{proof}[Proof of Theorem \ref{thm:VanHamme2}]
We retrace the steps in the proof of Theorem \ref{thm:CMpoints3plus} and choose 
$a,b,c',d'$ with $e|d,e|a,N|c',ad'-bc'=e$. We have seen that there exists an integer $k$
such that
$$
\frac{\alpha+k}{p}=\frac{a\alpha+b}{c'\alpha+d'}.
$$
We have also seen that the specialization of $x_k:=p^2\Theta(\tau)/\Theta((\tau+k)/p)$ at $\tau=\alpha$ reads
$\vartheta_p=\chi(w_e)(c\alpha+d)^2/e$, a $p$-adic unit. Next we compute the specialisation of the derivative
$\frac{dx_k}{dt}$ at $\tau=\alpha$.  Notice that
$$
\frac{dx_k}{dt}=\frac{d}{dt}\left(p^2\frac{\Theta(\tau)}{\Theta((\tau+k)/p)}\right)
=\frac{1}{t'(\tau)}\times\frac{p^2\Theta(\tau)}{\Theta((\tau+k)/p)}
\left(\frac{\Theta'(\tau)}{\Theta(\tau)}-\frac{1}{p}\frac{\Theta'((\tau+k)/p)}{\Theta((\tau+k)/p)}\right),
$$
where the dash ${}'$ denotes differentiation with respect to $\tau$. We specialize the right hand side
at $\tau=\alpha$. From the modularity of $\Theta$ we get
\[
e\Theta'\left(\frac{a\tau+b}{c'\tau+d'}\right)\left/\Theta\left(\frac{a\tau+b}{c'\tau+d'}\right)\right.
=2c'(c'\tau+d')+(c'\tau+d')^2\frac{\Theta'(\tau)}{\Theta(\tau)}
\]
Now substitute $\tau=\alpha$ to find
$$
e\frac{\Theta'((\alpha+k)/p)}{\Theta((\alpha+k)/p)}=
2c'(c'\alpha+d')+(c'\alpha+d')^2\frac{\Theta'(\alpha)}{\Theta(\alpha)}
$$
and hence,
$$
\frac{dx_k}{dt}(\alpha)=\frac{\chi(w_e)}{et'(\alpha)}(c\alpha+d)^2
\left(\frac{\Theta'(\alpha)}{\Theta(\alpha)}-\frac{1}{ep}2c'(c'\alpha+d')
-\frac{1}{ep}(c'\alpha+d')^2\frac{\Theta'(\alpha)}{\Theta(\alpha)}\right).
$$
Simplify this, using $c'\alpha+d'=c\overline\alpha+d$ and $ep=|c\alpha+d|^2$, to get
$$
\frac{dx_k}{dt}(\alpha)=\frac{\chi(w_e)}{et'(\alpha)}c(c\alpha+d)\left(2+(\alpha-\overline\alpha)
\frac{\Theta'(\alpha)}{\Theta(\alpha)}\right).
$$
Multiply by $t(\alpha)$ and use the definition of $\delta(\alpha)$ to get
\begin{equation}\label{derivative-xk}
t(\alpha)\frac{dx_k}{dt}(\alpha)=\frac{\chi(w_e)}{e}c(c\alpha+d)(\alpha-\overline\alpha)
\delta(\alpha).
\end{equation}
We now prove that $t(\alpha)\frac{dx_k}{dt}(\alpha)\is \theta F_p(t(\alpha))\mod{p^3}$,
where $\theta=\frac{d}{dt}$. 
 
Recall the polynomial $T_p(x,y)\in\Z[x,y]$ from Proposition \ref{theta-mirror3}.
We know that $T_p(x_k(\tau),t(\tau))=0$. After derivation with respect to $t$ we obtain
$\partial_xT_p(x_k,t)\frac{dx_k}{dt}+\partial_yT_p(x_k,t)=0$, where $\partial_xT_p,
\partial_yT_p$ are the partial derivatives of $T_p$. From this we see that $\frac{dx_k}{dt}$
lies in $\Q(x_k,t)$. After specialization of $\tau$ to $\alpha$ we find that 
$\frac{dx_k}{dt}(\alpha)\in\Q(\alpha,t(\alpha))$. And hence also $\delta(\alpha)\in\Q(\alpha,t(\alpha)$.
We now compute the specialization at $\tau=\alpha$ of $\partial_xT_p(x_k,t)$ and $\partial_yT_p(x_k,t)$
considered modulo $p^3$. 

From Proposition \ref{theta-mirror3} we know that
$$
\partial_xT_p(x_k,t)\is T_{p,p+1}(t)(p+1)x_k^p+T_{p,p}(t)px_k^{p-1}\mod{p^3}.
$$
From Corollary \ref{specialcongruence3} we know that $T_{p,p+1}(t(\alpha))\vartheta_p+T_{p,p}(t(\alpha))
\is0\mod{p^3}$. Hence $\partial_xT_p(\vartheta_p,t(\alpha))\is T_{p,p+1}(t(\alpha))\vartheta_p^p\mod{p^3}$.

Again using Proposition \ref{theta-mirror3} we find that
$$
\partial_yT_p(x_k,t)\is \partial_tT_{p,p+1}(t)x_k^{p+1}+\partial_tT_{p,p}(t)x_k^{p}\mod{p^3}.
$$
Recall from the proof of Theorem \ref{thm:functional3} that $T_{p,p+1}(t)F_p(t)+T_{p,p}(t)\is0\mod{p^3}$.
After derivation, 
$$
\partial_tT_{p,p+1}(t)F_p(t)+\partial_tT_{p,p}(t)\is-\partial_tF_p(t)T_{p,p+1}(t)\mod{p^3}.
$$
Multiply on both sides by $x_k^p$ and specialize $\tau$ at $\alpha$. Since $F_p(t(\alpha))\is\vartheta_p
\mod{p^3}$ we obtain
$$
\partial_tT_{p,p+1}(t(\alpha))\vartheta_p^{p+1}+\partial_tT_{p,p}(t(\alpha))\vartheta_p^p
\is-\partial_tF_p(t(\alpha))T_{p,p+1}(t(\alpha))\vartheta_p^p\mod{p^3}.
$$
The left hand side is precisely $\partial_yT_p(x_k,t)$ specialized at $\alpha$. Hence we conclude that
$$
T_{p,p+1}(t(\alpha))\vartheta_p^p\frac{dx_k}{dt}(\alpha)-\partial_tF_p(t(\alpha))T_{p,p+1}(t(\alpha))
\vartheta_p^p\mod{p^3}.
$$
Since $\vartheta_p$ is a $p$-adic unit and $T_{p,p+1}(t)\is1\mod{p}$ by Proposition \ref{theta-mirror3},
we conclude that
$$
\frac{dx_k}{dt}(\alpha)\is\partial_tF_p(t(\alpha))\mod{p^3}.
$$
Multiply by $t(\alpha)$ and invoke \eqref{derivative-xk}. After division by $-\delta(\alpha)$
we obtain
$$
-\frac{1}{\delta(\alpha)}\theta F_p(t(\alpha))\is \frac{\chi(w_e)}{e}
c(\overline\alpha-\alpha)(c\alpha+d)\mod{p^3}.
$$
Add the congruence $F_p(t(\alpha))\is\frac{\chi(w_e)}{e}(c\alpha+d)^2\mod{p^3}$
from Theorem \ref{thm:CMpoints3plus} and we get
$$
\left(1-\frac{1}{\delta(\alpha)}\theta\right)F_p(t(\alpha))\is
\frac{chi(w_e)}{e}(c\overline\alpha+d)(c\alpha+d)\is\chi(w_e)p\mod{p^3}.
$$
\end{proof}

\begin{proof}[Proof of Theorem \ref{thm:Ramanujan}]
Let us start with the identity $F(t(\tau))=\Theta(\tau)$. Take the derivation $\theta=t\frac{d}{dt}$
to get
$$
\theta F(t(\tau))=\frac{t(\tau)}{t'(\tau)}\Theta'(\tau)=\Theta(\tau)
\left(\delta(\tau)-\frac{2\pi it(\tau)}{t'(\tau)}\frac{1}{\pi i(\tau-\overline\tau)}\right).
$$
Multiply this eqality by $1/\delta(\tau)$ and subtract from $F(t(\tau))=\Theta(\tau)$.
We obtain
$$
\left(1-\frac{1}{\delta(\tau)}\theta\right)F(t(\tau))=t(\tau)\frac{2\pi i\Theta(\tau)}{t'(\tau)}
\frac{1}{\delta(\tau)}\frac{1}{\pi i(\tau-\overline\tau)},
$$
which gives Theorem \ref{thm:Ramanujan}. 
It is a straightforward computation to check that $\delta(\tau)$ is invariant under $\Gamma$. 
Similarly it is straightforward to check that $(2\pi i\Theta(\tau)/t'(\tau))^r$ is invariant under
$\Gamma$. Hence $(2\pi i\Theta(\tau)/t'(\tau))^{r}$ is a rational function of $t(\tau)$. 
\end{proof}

\section{Some extras}\label{sec:extras}
In this section we deal with some examples that go beyond the general theorems we have stated.
The first example is Theorem \ref{thm:CMpoints2plus}. 

\begin{proof}[Proof of Theorem \ref{thm:CMpoints2plus}]
We are in the case hypergeometric $F(1/2,1/2;1|16t)$ from Appendix B. In that case the modular
group is $\Gamma_0(4)$ but we extend it to $\Gamma_0(2)$, which normalizes $\Gamma_0(4)$
with index $2$. In particular we need the identy
$\Theta(1/(2\tau+1))=\pm(2\tau+1)\sqrt{t(\tau)}\Theta(\tau)$. Consequently we have
$\Theta(\frac{a\tau+b}{c\tau+d})=\pm(c\tau+d)\sqrt{t(\tau)}\Theta(\tau)$
for any integers $a,b,c,d$ with $ad-bc=1$ and $c\is2\mod{4}$.
We will not fight for the determination of the $\pm$-sign. 

We copy the proof Theorem \ref{thm:CMpoints2}.
Write $c\overline\alpha+d=c'\alpha+d'$, where $c'=-c,d'=d+c(\alpha+\overline\alpha)$.
Choose $a,b\in\Z$ such that $ad'-bc'=1$. Then, after some computation,
$$
\frac{a\alpha+b}{c'\alpha+d'}=\frac{1}{p}(\alpha+k)\quad\text{with}\quad
k=ac'|\alpha|^2+bc'(\alpha+\overline\alpha)+bd'.
$$
Since $c'\is2\mod4$ we get
$$
\chi(p)p\frac{\Theta(\alpha)}{\Theta((\alpha+k)/p)}=
\frac{\chi(p)p}{\chi(d')(c'\alpha+d')\sqrt{t(\alpha)}}.
$$
Using $p/(c'\alpha+d')=c\alpha+d$ our theorem follows as a consequence of Corollary
\ref{specialcongruence2}. 
\end{proof}

The second set of examples concern an extension of Theorem \ref{thm:functional3} and its
consequences. 

\begin{theorem}\label{thm:functional3extend}
Suppose that Assumption (II) holds, except for part (4) where we assume that $r=1/3$.
Then we have
$$
F(t)\is F(t^\sigma)(F_p(t)+(g_p-g_1)t^p)\mod{p^3},
$$
where $g_k$ is the coefficient of $t^k$ in the power series $F(t)$.
\end{theorem}

\begin{proof}
The proof is the same as the proof of Theorem \ref{thm:functional3}, except at the end, where we
must conclude that $F(t)/F(t^\sigma)$ equals a polynomial of degree $\le p$ modulo $p^3$.  
Recall that $t^\sigma=t^p+O(t^{p+1})$. Hence $F(t^\sigma)\is 1+g_1t^p\mod{t^{p+1}}$ 
and $F(t)/F(t^\sigma)\is F_p(t)+(g_p-g_1)t^p$. 
\end{proof}

\begin{corollary}
Consider the case hypergeometric $F(1/3,1/2,2/3;1,1|108t)$. Then for all primes $p>3$,
$$
F(t)\is F(t^\sigma)F_p(t)\mod{p^3}.
$$
\end{corollary}

\begin{proof}
We apply Theorem \ref{thm:functional3extend}. We need to verify that $g_p\is g_1\mod{p^3}$
with $g_k=(3k)!/k!^3$. This congruence follows from Wolstenholme's congruence
$\binom{3p}{2p}\binom{2p}{p}^2\is \binom{3}{2}\binom{2}{1}^2\mod{p^3}$. 
\end{proof}

We finally like to extend Theorem \ref{thm:functional3} to the case 
hypergeometric $F(1/6,1/2,5/6;1,1|1728t)$. In that case $\Theta$ is not a modular form, but 
$\Theta^2=E_4$ is the weight 4 modular form with respect to $SL(2,\Z)$. 
We apply our ideas to the series $\tilde{F}=F(1/6,1/2,5/6;1,1|1728t)^2$, which satisfies a
differential equation of order $6$. Then, clearly, $\tilde{\Theta}=E_4$ and $t(\tau)$ is still
$1/J(\tau)$. Notice also that $1/J$ has its pole at $\tau=(1+\sqrt{-3})/2$ and $E_4$ has a 
zero of order $1/3$ at that point. So $r=1/3$. 

\begin{lemma}\label{lemma:fullmodular}
Denote the coefficients of $\tilde{F}(t)$ by $\tilde{g}_k$ and the $p$-trucation of
$\tilde{F}(t)$ by $\tilde{F}_p(t)$. Then $\tilde{g}_p\is 2g_p\mod{p^3}$ and $\tilde{g}_1=2g_1$.
Moreover, $g_p\is g_1\mod{p^3}$. 
Finally,
$$
\tilde{F}_p(t)\is F_p(t)^2\mod{p^3}.
$$
\end{lemma}

\begin{proof}
Let $p$ be a prime $>3$.
By examination of the coefficients $g_k=\binom{6k}{3k}\binom{3k}{k}\binom{2k}{k}$ one finds that $g_k$ is
divisible by $p$ when $p/2>k>p/6$, divisible by $p^2$ when $5p/6>k>p/2$ and divisible by $p^3$ when $p>k>5p/6$.
Consequently $F_p(t)^2$ equals modulo $p^3$ a polynomial of degree $<p$. So this must be $\tilde{F}_p(t)$.

Secondly, $\tilde{g}_p=\sum_{k=0}^p g_kg_{p-k}$. By our remark on divisibility of $g_k$ by powers
of $p$, we find that all terms of this summation are 0 modulo $p^3$,
except when $k=0,p$. This gives us $\tilde{g}_p\is 2g_p\mod{p^3}$.
The equality $\tilde{g}_1=2g_1$ is immediate. 

Thirdly, the congruence
$g_p\is g_1\mod{p^3}$ follows from Wolstenholme's theorem for binomial coefficients.
\end{proof}

\begin{proof}[Proof of Theorem \ref{thm:RamanujanVanHamme}]
Analogous to the proof of Theorem \ref{thm:functional3} we construct a polynomial $T_p(x,y)$
such that $T_p(E_4(\tau)/E_4(p\tau),t(\tau))=0$ and $T_p(x,y)\is A(y)x^p+B(y)x^{p+1}\mod{p^3}$
with $\deg(A)\le p/3$ and $\deg(B)\le(p+1)/3$. Hence, repeating the proof of Theorem
\ref{thm:functional3extend}, we find that
$$
\tilde{F}(t)\is \tilde{F}(t^\sigma)(\tilde{F}_p(t)+(\tilde{g}_p-\tilde{g}_1)t^p)\mod{p^3},
$$
where $\tilde{g}_k$ is the $k$-th coefficient of $\tilde{F}(t)$. Using Lemma \ref{lemma:fullmodular}
we see that the coefficient of $t^p$ is $0\mod{p^3}$. By the same lemma we also replace $\tilde{F}_p(t)$
by $F_p(t)^2$ modulo $p^3$. Hence we find
$$
F(t)^2\is F(t^\sigma)^2F_p(t)^2\mod{p^3}.
$$
As power series in $t$ we can conclude that $F(t)\is F(t^\sigma)F_p(t)\mod{p^3}$. For the specialization
step like in Corollary \ref{specialcongruence3} we need to retain the squared version. 
The analogue of this corollary yields a statement of the form  $F_p(t(\tau_0))^2\is\vartheta_p\mod{p^3}$,
where $\vartheta_p$ is the unit root value among $E_4(\tau_0)/E_4(p\tau_0)$ and
$p^4E_4(\tau_0)/E_4((\tau_0+k)/p)$ with $k=0,\ldots,p-1$. 

For any $p$ that splits in $\Q(\sqrt{-D})$ we write $p=|c\omega_D+d|^2$. Note $\wp=c\omega_D+d$.
Choose $a,b\in\Z$ such that $ad-bc=1$ and set $\tau_0=\frac{a\omega_D+b}{c\omega_D+d}$. 
Clearly $t(\tau_0)=t(\omega_D)$ and $E_4(\tau_0)=(c\omega_D+d)^4E_4(\omega_D)$. But also 
$p\tau_0=\omega_D+k$ for some integer $k$, hence $E_4(p\tau_0)=E_4(\omega_D)$. Thus we conclude
that $F(t(\omega_D))^2\is \wp^4\mod{p^3}$. The first congruence statement now follows. 
The Van Hamme-type congruence follows in the same way as in the proof of Theorem \ref{thm:VanHamme2}.
\end{proof}

\section{Appendix A}
Modular subgroups and their Hauptmoduln.
We use the standard notations
$$\eta_N=\eta(q^N)\quad\text{where}\quad\eta(q)=q^{1/24}\prod_{k>0}(1-q^k)$$
$$E_{2,N}=E_2(q)-NE_2(q^N)\quad\text{where}\quad E_2(q)=1-24\sum_{k>0}\frac{kq^k}{1-q^k}.$$
 
In the following listing we choose the Hauptmodul $h_N$ for each level $N$ 
in such a way that $h_N(q)=q+O(q^2)$ and $w_Nh_N(\tau)h_N(-1/N\tau))=1$ for some integer $w_N$. 

A slightly more extended list can be found in \cite[Table 2]{Ma09}. The cusps and
elliptic points are listed by their $h_N$-values. 
\medskip

\begin{center}
\begin{tabular}{|c|c|c|l|l|}
\hline
$N$ & $h_N$ & $w_N$ & cusps & elliptic points\\
\hline
$2$ & $\eta_2^{24}/\eta_1^{24}$ & $4096$ & $\infty,0$ & $-1/64$ (quadratic)\\
$3$ & $\eta_3^{12}/\eta_1^{12}$ & $729$ & $\infty,0$ & $-1/27$ (cubic)\\
$4$ & $\eta_4^8/\eta_1^8$ & $256$ & $\infty,-1/16,0$ & \\
$5$ & $\eta_5^6/\eta_1^6$ & $125$ & $\infty,0$ & $(-11\pm2\sqrt{-1})/125$ (quadratic)\\
$6$ & $\eta_2\eta_6^5/\eta_3\eta_1^5$ & $72$ & $\infty,0,-1/8,-1/9$ & \\
$7$ & $\eta_7^4/\eta_1^4$ & $49$ & $\infty,0$ & $(-13\pm3\sqrt{-3})/98$ (cubic)\\
$8$ & $\eta_2^2\eta_8^4/\eta_1^4\eta_4^2$ & $32$ & $\infty,-1/4,-1/8,0$ & \\
$9$ & $\eta_9^3/\eta_1^3$ & $27$ & $\infty,(-9\pm3\sqrt{3})/54,0$ & \\
$10$ & $\eta_2\eta_{10}^3/\eta_1^3\eta_5$ & $20$ & $\infty,-1/5,-1/4,0$ & $(-4\pm2\sqrt{-1})/20$
(quadratic)\\
$12$ & $\eta_2^2\eta_3\eta_{12}^3/\eta_1^3\eta_4\eta_6^2$ & $12$ &
6 cusps & \\
$18$ & $\eta_2\eta_3\eta_{18}^2/\eta_1^2\eta_6\eta_9$ & $6$ & 8 cusps & \\
\hline
\end{tabular}
\end{center}
\medskip

\section{Appendix B: second order equations}
In this section we list second order hypergeometric differential equations
whose monodromy group is defined over $\Z$, and Ap\'ery-like equations whose numbering is
taken from Zagier's paper \cite{Zagier07}. We also list the mirror map and the
associated form $\Theta$. In some entries we added the adjective 'twisted'.
This means that we deviated from the standard literature by choosing the coefficients
$g_n$ as $(-1)^n$ times the standard coefficients. This simplies the modular group
in such cases. Furthermore we list (hopefully complete) rational 
CM-values, for use as argument in the supercongruence applications. 
Note that in the case hypergeometric $F(1/4,3/4;1|64t)$
the function $\Theta$ is not a modular form (its square is).  In that case
one still finds many supercongruences, which can be shown in a manner parallel
to the proof of Theorem \ref{thm:RamanujanVanHamme}. Note that for hypergeometric
$F(1/6,5/6;1|864t)$ the function $t$ is not modular.

\subsection{Hypergeometric $F(1/2,1/2,1|16t)$}
Differential operator $\theta_t^2 - 4t (2\theta_t+1)^2$
$$
\text{Coefficients:}\quad g_n=\binom{2n}{n}^2
$$
$$\text{Mirror map: }t(q)=\frac{\eta_4^{16}\eta_1^8}{\eta_2^{24}}=\frac{h_4}{1+16h_4}=-h_4(-q),
\quad\text{pole: }1/2$$
$$\Theta(q)=\theta_S=\frac{\eta_2^{10}}{\eta_1^4\eta_4^4},\quad\text{Group: }\Gamma_0(4)$$
$$t(q)\Theta(q)^2=\frac{\eta_4^8}{\eta_2^4}
=\left(\sum_{n\in\Z}q^{n(n+1)/2}\right)^4$$
$$\text{Cusp values:}\quad t(0)=1/16,\quad t(1/2)=\infty$$
$$\text{Rational CM-values:}\quad t(i/2)=\frac{1}{32},\quad t((i+1)/2)=-\frac{1}{16},
\quad t((1+i)/4)=\frac{1}{8}$$
$$\text{Some quadratic CM-values:}\quad t(i)=(1-\sqrt{2})^4/16,\quad t((2+i)/5)=(1+\sqrt{2})^4/16$$
$$t(\sqrt{-2}/2)=(1-\sqrt{2})^2/16,\quad t((2+\sqrt{-2})/6)=(1+\sqrt{2})^2/16$$

\subsection{Hypergeometric $F(1/3,2/3,1|27t)$}
Differential operator $\theta_t^2 - 3t (3\theta_t+1)(3\theta_t+2)$
$$
\text{Coefficients:}\quad g_n=\frac{(3n)!}{n!^3}
$$
$$t(q)=\frac{h_3}{1+27h_3},\quad\text{pole:}\quad(3+\sqrt{-3})/6$$
$$\Theta(q)=\theta_H=1+6\sum_{n\ge1}\chi_{-3}(n)\frac{q^n}{1-q^n},\quad\text{Group: }\Gamma_0(3)$$
$$t(q)\Theta(q)^3=\frac{\eta_3^9}{\eta_1^3}=
\sum_{n\ge1}\sum_{d|n}d^2\chi_{-3}(n/d)q^n$$

$$\text{Rational cusp values and CM-values:}\quad t(0)=\frac{1}{27},
\quad t(\sqrt{-3}/3)=\frac{1}{54},$$
$$t((1+\sqrt{-3})/2)=-\frac{1}{216},\quad t((-1+\sqrt{-3})/6)=\frac{1}{24}$$

\subsection{Hypergeometric $F(1/4,3/4,1|64t)$}

Differential operator $\theta_t^2 - 4t (4\theta_t+1)(4\theta_t+3)$

$$t(q)=\frac{h_2}{1+64h_2}\quad\text{pole:}\quad (1+i)/2$$

$$\Theta(q)^2=-E_{2,2}$$

$$t(q)\Theta(q)^4=\sum_{n\ge1}\sum_{d|n}d^3\chi_2(n/d)q^n,\quad\text{where $\chi_2(m)=1$ if $m$ is odd,
and $0$ if $m$ is even}$$
$$\text{Rational CM-values:}\quad t((1+\sqrt{-7})/2)=-\frac{1}{4032}
\quad t((-1+\sqrt{-7})/8)=\frac{1}{63},$$
$$t(i)=\frac{1}{576},\quad t(i/2)=\frac{1}{72},\quad t(\sqrt{-2}/2)=\frac{1}{128},$$
$$t((1+\sqrt{-3})/2)=-\frac{1}{192},\quad t((-1+\sqrt{-3})/4)=\frac{1}{48}$$

\subsection{Hypergeometric $F(1/6,5/6,1|864t)$}

Differential operator $\theta_t^2 - 24t (6\theta_t+1)(6\theta_t+5)$

$$t(q)=\frac{E_4^{3/2}-E_6}{864E_4^{3/2}}$$

$$J=\frac{1}{t(q)-432t(q)^2},\quad\Theta(q)^4=E_4$$

\subsection{Zagier A (Franel numbers)}
Differential operator $\theta_t^2 - t(7\theta_t^2 + 7\theta_t + 2)-8t^2(\theta_t + 1)^2$
$$
\text{Coefficients:}\quad g_n=\sum_{k=0}^n\binom{n}{k}^3
$$
$$\text{Mirror map:}\quad t(q)=\frac{\eta_1^3\eta_6^9}{\eta_2^3\eta_3^9}=\frac{h_6}{1+8h_6},
\quad\text{Group: }\Gamma_0(6)$$ 
$$\Theta(q)=\frac13\theta_H(q)+\frac23\theta_H(q^2)=\frac{\eta_2\eta_3^6}{\eta_1^2\eta_6^3}$$
$$\text{Cusp values:}\quad t(0)=\frac18,\quad t(1/2)=-1$$
$$\text{Rational CM-values:}\quad t((3+\sqrt{-3})/6)=-\frac14,\quad t((3+\sqrt{-3}/12)=\frac12$$

\subsection{Zagier B} (twisted)
Differential operator $\theta_t^2 + t(9\theta_t^2 + 9\theta_t + 3) + 27t^2(\theta_t + 1)^2$
$$
\text{Coefficients:}\quad g_n=\sum_{k=0}^{\lfloor n/2\rfloor}(-3)^{n-3k}\frac{n!}{(n-3k)!k!^3}
$$
$$\text{Mirror map:}\quad t(q)=\frac{\eta_9^3}{\eta_1^3}=h_9,\quad\text{Group: }\Gamma_0(9)$$
$$\Theta(q)=\frac{\eta_1^3}{\eta_3}=-\frac12\theta_H(q)+\frac32\theta_H(q^3)$$

$$\text{Rational CM-value:}\quad t((3+\sqrt{-3})/6)=\frac19\quad t((3+\sqrt{-3})/18)=\frac13$$

\subsection{Zagier C}
Differential operator $\theta_t^2 - t(10\theta_t^2 + 10\theta_t + 3) + 9t^2(\theta_t + 1)^2$
$$
\text{Coefficients:}\quad g_n=\sum_{k=0}^n\binom{n}{k}^2\binom{2k}{k}
$$
$$\text{Mirror map:}\quad t(q)=\left(\frac{\eta_6^2\eta_1}{\eta_2^2\eta_3)}\right)^4=\frac{h_6}{1+9h_6}$$
$$\Theta(q)=\frac{\eta_3\eta_2^6}{\eta_1^3\eta_6^2},\quad\text{Group: }\Gamma_0(6)$$
$$\text{Cusp values:}\quad t(0)=1/9,\quad t(1/3)=1$$
$$\text{Rational CM-values:}\quad t((3+\sqrt{-3})/6)=-\frac13,\quad t((3+\sqrt{-3}/12)=\frac13$$

\subsection{Zagier D (Ap\'ery numbers)}
Differential operator $\theta_t^2 - t (11 \theta_t^2 + 11 \theta_t + 3) - t^2 (\theta_t+1)^2 $
$$
\text{Coefficients:}\quad g_n=\sum_{k=0}^n\binom{n}{k}^2\binom{n+k}{k}
$$
$$\text{Mirror map:}\quad t(q)=q\prod_{n>0}(1-q^n)^{5\left(\frac{n}{5}\right)},\quad\text{Group: }\Gamma_1(5)$$
$$\Theta(q)=1+\frac12\sum_{k>0}((3-i)\chi(k)+(3+i)\overline\chi(k))\frac{q^k}{1-q^k},\quad \chi(2)=i$$

$$t(q)\Theta(q)^2=\frac{\eta_5^5}{\eta_1}$$

\subsection{Zagier E}
Differential operator $\theta_t^2 - t(12\theta_t^2 + 12\theta_t + 4) + 32t^2(\theta_t + 1)^2$
$$
\text{Coefficients:}\quad g_n=\sum_{k=0}^{\lfloor n/2\rfloor}4^{n-2k}\binom{n}{2k}\binom{2k}{k}^2
$$
$$\text{Mirror map:}\quad t(q)=\left(\frac{\eta_1^2\eta_4\eta_8^2}{\eta_2^5}\right)^2=\frac{h_8}{1+8h_8}$$ 
$$\Theta(q)=\theta_S=\frac{\eta_2^{10}}{\eta_1^4\eta_4^4},\quad\text{Group: }\Gamma_0(8)$$
$$\text{Cusp values:}\quad t(0)=\frac18,\quad t(1/4)=\frac14$$

\subsection{Zagier F}(twisted)
Differential operator: $\theta^2 + t(17\theta_t^2 + 17\theta_t + 6) + 72t^2(\theta_t + 1)^2$
$$
\text{Coefficients:}\quad g_n=\sum_{k=0}^n(-8)^{n-k}\binom{n}{k}\sum_{j=0}^k\binom{k}{j}^3
$$
$$\text{Mirror map:}\quad t(q)=\frac{\eta_6^5\eta_2}{\eta_1^5\eta_3}=h_6,\quad\text{Group: }\Gamma_0(6)$$
$$\Theta(q)=\frac{\eta_1^6\eta_6}{\eta_2^3\eta_3^2}=-\theta_H(q)+2\theta_H(q^2)$$
$$\text{Cusp values:}\quad t(1/2)=-\frac19,\quad t(1/3)=-\frac18$$
$$\text{Rational CM-values:}\quad t((3+\sqrt{-3})/6)=-\frac{1}{12},
\quad t((3+\sqrt{-3})/12)=-\frac16$$

\section{Appendix C: third order differential equations}
This section lists third order equations and their corresponding mirror maps,
modular forms and rational CM-values taken from the following sources.
First of all the hypergeometric equations, whose monodromy group is defined over $\Z$.
Then a list of Ap\'ery like equations
which are described in Almkvist and Zudilin's paper \cite{AZ06}. Their numbering 
by Greek letters is adopted in this paper. In \cite{AZ06} there is also an extensive description
of their relation with second order equations. Furthermore we listed the sporadic 
sequences Cooper $s_7,s_{10},s_{18}$ found by Cooper in \cite{Cooper12}.
The last two examples are Hadamard
products of the central binomial coefficient $\binom{2n}{n}$ with the
coefficients of the solution of a second order equation. We take the Ap\'ery numbers
and Franel numbers for the latter. The corresponding supercongruences are quite popular
in Z.W.Sun's list \cite{ZW-Sun19} and the papers \cite{ZH-Sun18} and \cite{ZH-Sun21}.

Many of the formulas in this section can also be found in Cooper's book \cite{Cooper17}.
We remark that in the case hypergeometric $F(1/6,1/2,5/6;1,1|1728t)$ the function $\Theta$ is
not modular (but its square is). Nevertheless there are some very interesting supercongruences
associated to it, see Theorem \ref{thm:RamanujanVanHamme}.

\subsection{Hypergeometric $F(1/2,1/2,1/2;1,1|64t)$}
Differential operator: $\theta^3-4t(2\theta+1)^3$
$$
\text{Coefficients:}\quad g_n=\binom{2n}{n}^3.
$$
$$
\text{Mirror map:}\quad t(q)=\left(\frac{\eta_1\eta_4}{\eta_2^2}\right)^{24}=
\frac{h_4}{(1+16h_4)^2},
\quad\text{pole: }1/2
$$
$$
\Theta(q)=\theta_S^2,\quad t\Theta^4=\left(\frac{\eta_1\eta_4}{\eta_2}\right)^8,
\quad\text{Group: }\Gamma_0(4)+\left<4\right>
$$
$$
\text{Rational CM-values:}\quad t(i/2)=\frac{1}{64},\quad t((1+\sqrt{-2})/2)=-\frac{1}{64},
\quad t(\sqrt{-3}/2)=\frac{1}{256},
$$
$$
t((1+\sqrt{-3})/4)=\frac{1}{16},\quad t((1+i)/2)=-\frac12,\quad t(1/2+i)=-\frac{1}{512},
$$
$$
t((3+\sqrt{-7})/8)=1,\quad t(\sqrt{-7}/2)=\frac{1}{4096}
$$

\subsection{Hypergeometric $F(1/4,1/2,3/4;1,1|256t)$}\ \\
Differential operator: $\theta^3-8t(4\theta+1)(2\theta+1)(4\theta+3)$
$$
\text{Coefficients:}\quad g_n=\frac{(4n)!}{n!^4}.
$$
$$
\text{Mirror map:}\quad t(q)=\frac{h_2}{(1+64h_2)^2}\quad\text{pole: $(1+i)/2$}
$$
$$
\Theta(q)=-E_{2,2},\quad t\Theta^4=(\eta_1\eta_2)^8,\quad\text{Group: }\Gamma_0(2)+\left<2\right>
$$
$$
\text{Rational CM-values:}\quad t(\sqrt{-1/2})=\frac{1}{256},
\quad t((1+\sqrt{-5})/2)=-\frac{1}{1024},\quad t(\sqrt{-3/2})=\frac{1}{2304},
$$
$$
t((1+\sqrt{-3})/2)=-\frac{1}{144}, \quad t((1+3i)/2)=-\frac{1}{17088},
\quad t(i)=\frac{1}{648},
$$
$$
t(\sqrt{-5/2})=\frac{1}{20736},\quad t((1+\sqrt{-7})/4)=\frac{1}{81},
\quad t((1+\sqrt{-13})/2)=-\frac{1}{82944},
$$
$$
t((1+\sqrt{-7})/2)=-\frac{1}{3969},\quad t(3\sqrt{-2}/2)=\frac{1}{614656},
\quad t(\sqrt{-11/2})=\frac{1}{9801},
$$
$$
t((1+5i)/2)=-\frac{1}{6635520},\quad t((1+\sqrt{-37})/2)=-\frac{1}{199148544},
\quad t(\sqrt{-29/2})=\frac{1}{24591257856}
$$

\subsection{Hypergeometric $F(1/3,1/2,2/3;1,1|108t)$}\ \\
Differential operator: $\theta^3-6t(3\theta+1)(2\theta+1)(3\theta+2)$
$$
\text{Coefficients:}\quad g_n=\frac{(3n)!(2n)!}{n!^5}.
$$
$$
\text{Mirror map:}\quad t(q)=\frac{h_3}{(1+27h_3)^2}\quad\text{pole: $(3+\sqrt{-3})/6$}
$$
$$
-2\Theta(q)=E_{2,3},\quad t\Theta^3=(\eta_1\eta_3)^6,\quad\text{Group: }\Gamma_0(3)+\left<3\right>
$$
$$
\text{Rational CM-values:}\quad t(\sqrt{-3}/3)=\frac{1}{108},
\quad t((3+\sqrt{-15})/6)=-\frac{1}{27},\quad t(\sqrt{-6}/3)=\frac{1}{216},
$$
$$ 
t((1+\sqrt{-2})/3)=\frac{1}{8},
\quad t((1+\sqrt{-17/3})/2)=-\frac{1}{1728},\quad t((1+\sqrt{-3})/2)=-\frac{1}{192},$$
$$t(1+\sqrt{-11})/6)=\frac{1}{1024},\quad t(1+5\sqrt{-1/3})/2)=-\frac{1}{8640},
\quad t(\sqrt{-5/3})=\frac{1}{3375},$$
$$t((1+\sqrt{-41/3})/2)=-\frac{1}{110592},
\quad t((1+7\sqrt{-1/3})/2)=-\frac{1}{326592}
$$

\subsection{Hypergeometric $F(1/6,1/2,5/6;1,1|1728t)$}\ \\
Differential operator: $\theta^3-24t(6\theta+1)(2\theta+1)(6\theta+5)$
$$
\text{Coefficients:}\quad g_n=\frac{(6n)!}{(3n)!n!^3}.
$$
$$
\text{Mirror map:}\quad t(q)=\frac{1}{J},\quad\text{Group: }\Gamma_0(1)
$$
$$
\Theta(q)^2=E_4,\quad t\Theta^6=12^3\Delta
$$

\subsection{Ap\'ery$_3$, Almkvist ($\gamma$)}\ \\
Differential operator: $(1 - 34 t + t^2) \theta^3+(-51 t + 3 t^2) \theta^2+ (-27 t + 3 t^2) \theta -5 t + t^2 $ 
$$
\text{Coefficients:}\quad g_n=\sum_{k=0}^n\binom{n}{k}^2\binom{n+k}{k}^2.
$$
$$\text{Mirror map:}\quad 
t(q)=\left(\frac{\eta_1\eta_6}{\eta_2\eta_3}\right)^{12}=\frac{h_6}{(1+8h_6)(1+9h_6)},
\quad\text{pole: }1/2$$
$$\Theta(q)=\frac{(\eta_2\eta_3)^7}{(\eta_1\eta_6)^5},\quad\text{Group: }\Gamma_0(6)+\left<6\right>$$
$$
\text{Rational CM-values:}\quad t((2+\sqrt{-2})/6)=1,\quad t((3+\sqrt{-3})/6)=-1
$$

\subsection{Almkvist ($\eta$)}(twisted)\ \\
Differential operator: 
$(1 + 22 t + 125 t^2) \theta^3 +(33 t + 375 t^2) \theta^2 +(21 t + 375 t^2) \theta + 5 t + 125 t^2	$
$$
\text{Coefficients:}\quad g_n=\sum_{k=0}^n(-1)^{n-k}\binom{n}{k}^3\binom{4n-5k}{3n}.
$$
$$\text{Mirror map:}\quad t(q)=\left(\frac{\eta_5}{\eta_1}\right)^6$$
No rational CM-values found.
$$\Theta(q)=\frac{\eta_1^5}{\eta_5},\quad\text{Group: }\Gamma_0(5)$$

\subsection{Domb, Almkvist ($\alpha$)}(twisted)\ \\
Differential operator: 
$(1 + 20 t + 64 t^2) \theta^3+(30 t + 192 t^2) \theta^2+(18 t + 192 t^2)
\theta + 4 t + 64 t^2$
$$
\text{Coefficients:}\quad g_n=(-1)^n\sum_{k=0}^n\binom{n}{k}^2\binom{2k}{k}\binom{2n-2k}{n-k}.
$$

$$\text{Mirror map:}\quad t(q)=\left(\frac{\eta_2\eta_6}{\eta_1\eta_3}\right)^{6}=
\frac{h_6(1+9h_6)}{1+8h_6}$$

$$\Theta(q)=\sum_{x,y,z,u\in\Z}(-q)^{x^2+y^2+3z^2+3u^2}=
\frac{\eta_1^4\eta_3^4}{\eta_2^2\eta_6^2},\quad\text{Group: }\Gamma_0(6)+\left<3\right>$$
$$
\text{Rational CM-values:}\quad t((9+\sqrt{-15})/24)=1,\quad t((3+\sqrt{-3})/6)=-\frac12,
\quad t(3+\sqrt{-3})/12)=\frac14,
$$
$$
t(1+\sqrt{-2})/6)=\frac18,\quad t((1+\sqrt{-2/3})/2)=-\frac18,\quad t(\sqrt{-3}/6)=\frac{1}{16},
$$
$$
t((1+\sqrt{-4/3})/2)=-\frac{1}{32},\quad t(\sqrt{-15}/30)=\frac{1}{64}
$$

\subsection{Almkvist-Zudilin, Almkvist ($\delta$)}(twisted)\ \\
Differential operator: 
$(1 + 14 t + 81 t^2) \theta^3 + (21 t + 243 t^2) \theta^2 + (13 t + 243 t^2) 
\theta + 3 t + 81 t^2$
$$
\text{Coefficients:}\quad g_n=\sum_{k=0}^{\lfloor n/3\rfloor}(-3)^{n-3k}\frac{(n+k)!}{k!^4(n-3k)!}.
$$

$$\text{Mirror map:}\quad t(q)=\left(\frac{\eta_3\eta_6}{\eta_1\eta_2}\right)^4
=\frac{h_6(1+8h_6)}{1+9h_6}$$
$$\Theta(q)=\frac{\eta_1^3\eta_2^3}{\eta_3\eta_6},\quad\text{Group: }\Gamma_0(6)+\left<2\right>$$
$$
\text{rational CM-values:}\quad t((3+\sqrt{-2})/6)=-1,\quad t((2+i)/6)=\frac13,
\quad t((3+\sqrt{-3})/12)=\frac19,
$$
$$
t((1+\sqrt{-2/3})/2)=-\frac19, \quad t(i/2)=\frac{1}{27},\quad t((1+\sqrt{-2})/2)=-\frac{1}{81}
$$

\subsection{Almkvist ($\epsilon$)}\ \\
Differential operator: 
$(1 - 24 t + 16 t^2) \theta^3 + (-36 t + 48 t^2) \theta^2 + (-20 t + 48 t^2)\theta -4 t + 16 t^2$
$$
\text{Coefficients:}\quad g_n=\sum_{k=\lceil n/2\rceil}^n\binom{n}{k}^2\binom{2k}{n}^2.
$$

$$\text{Mirror map:}\quad t(q)=\left(\frac{\eta_1\eta_8}{\eta_2\eta_4}\right)^8
=\frac{h_8}{(1+4h_8)(1+8h_8)},\quad\text{pole: }1/2$$
$$\Theta(q)=\frac{\eta_2^6\eta_4^6}{\eta_1^4\eta_8^4},\quad\text{Group: }\Gamma_0(8)+\left<8\right>$$
$$
\text{Rational CM-values:}\quad t((5+\sqrt{-7})/16)=1,\quad t((2+\sqrt{-2})/4)=-\frac14,
$$
$$
t((1+i)/4)=\frac14,\quad t((1+\sqrt{-7})/8)=\frac{1}{16}
$$

\subsection{Almkvist ($\zeta$)}\ \\
Differential operator: 
$(1 - 18 t - 27 t^2) \theta^3 + (-27 t - 81 t^2) \theta^2 + (-15 t - 81 t^2) \theta -3 t - 27 t^2$
$$
\text{Coefficients:}\quad g_n=\sum_{k,\ell}\binom{n}{k}^2\binom{n}{\ell}\binom{k}{\ell}\binom{k+l}{n}.
$$

$$\text{Mirror map:}\quad t(q)=\left(\frac{\eta_1\eta_9}{\eta_3^2}\right)^6
=\frac{h_9}{1+9h_9+27h_9^2},\quad\text{pole: }1/3$$
$$\Theta(q)=\frac{\eta_3^{10}}{\eta_1^3\eta_9^3},\quad\text{Group: }\Gamma_0(9)+\left<9\right>$$
$$
\text{Rational CM-values:}\quad t((5+\sqrt{-11})/18)=1,\quad t((4+\sqrt{-2})/9)=-1,
\quad t((1+\sqrt{-1/3})/2)=-\frac13,
$$
$$
t((1+\sqrt{-3})/6)=\frac19,\quad t(\sqrt{-2}/3)=\frac{1}{27},\quad t((3+\sqrt{-11})/6)=-\frac{1}{27}
$$

\subsection{Cooper s$_7$}\ \\
Differential operator: 
$(1 - 26 t - 27 t^2) \theta^3 + (-39 t - 81 t^2) \theta^2 + (-21 t - 78 t^2) \theta -4 t - 24 t^2$
$$
\text{Coefficients:}\quad g_n=\sum_{k=\ceil{n/2}}^n\binom{n}{k}^2\binom{n+k}{k}\binom{2k}{n}.
$$

$$t(q)=\frac{h_7}{1+13h_7+49h_7^2},\quad\text{pole: }(5+\sqrt{-3})/14$$
$$\Theta(q)=\left(1+2\sum_{k>0}\left(\frac{k}{7}\right)\frac{q^k}{1-q^k}\right)^2,
\quad t\Theta^{3/2}=(\eta_1\eta_7)^3,\quad\text{Group: }\Gamma_0(7)+\left<7\right>$$
$$
\text{Rational CM-values:}\quad t((1+\sqrt{-1/7})/2)=-1,\quad t((2+\sqrt{-3})/7)=\frac12,
\quad t((3+\sqrt{-19})/14)=\frac18,
$$
$$
t((1+\sqrt{-5/7})/2)=-\frac18,\quad t((1+3\sqrt{-3})/14)=\frac{1}{24},\quad t(\sqrt{-1/7})=\frac{1}{27},
$$
$$
t((1+\sqrt{-13/7})/2)=-\frac{1}{64},\quad t(\sqrt{-4/7})=\frac{1}{125},
\quad t((1+\sqrt{-61/7})/2)=-\frac{1}{10648}
$$

\subsection{Yang-Zudilin, Cooper s$_{10}$}\ \\
Differential operator: 
$(1 - 12 t - 64 t^2) \theta^3 + (-18 t - 192 t^2) \theta^2 + (-10 t - 188 t^2) \theta -2 t - 60 t^2$
$$
\text{Coefficients:}\quad g_n=\sum_{k=0}^n\binom{n}{k}^4.
$$

$$\text{Mirror map:}\quad t(q)=h_{10}\frac{(1+4h_{10})(1+5h_{10})}{(1+8h_{10}+20h_{10}^2)^2}$$
$$4\Theta(q)=E_{2,2}+E_{2,5}-E_{2,10},\quad t^3\Theta^4=(\eta_1\eta_2\eta_5\eta_{10})^4,
\quad\text{Group: }\Gamma_0(10)+\left<2,5\right>
$$
$$
\text{Rational CM-values:}\quad t((5+\sqrt{-15})/20)=1,
\quad t((2+i)/5)=-\frac12,\quad t((2+\sqrt{-6})/10)=\frac14,
$$
$$
t((1+\sqrt{-1/5})/2)=-\frac14,\quad t((1+\sqrt{-3/5})/2)=-\frac19,\quad t(1+3i)/10)=\frac{1}{12},
$$
$$
t(\sqrt{-1/10})=\frac{1}{16},\quad t((1+i)/2)=-\frac{1}{20},\quad t(\sqrt{-3/10})=\frac{1}{36},
$$
$$
t((1+\sqrt{-9/5})/2)=-\frac{1}{64},\quad t(\sqrt{-7/10})=\frac{1}{196},
\quad t((1+\sqrt{-17/5})/2)=-\frac{1}{324},
$$
$$
t(\sqrt{-13/10})=\frac{1}{1296},\quad t(\sqrt{-19/10})=\frac{1}{5776}
$$

\subsection{Cooper s$_{18}$}\ \\
Differential operator: 
$(1 - 28 t + 192 t^2) \theta^3 + (-42 t + 576 t^2) \theta^2 + (-26 t + 564 t^2) \theta -6 t + 180 t^2$
$$
\text{Coefficients:}\quad g_n=\sum_{k=0}^{\lfloor n/2\rfloor}\binom{n}{k}\binom{2k}{k}
\binom{2n-2k}{n-k}\left(\binom{2n-3k-1}{n}+\binom{2n-3k}{n}\right).
$$

$$
\text{Mirror map:}\quad t(q)=
\frac{h_{18}(1+2h_{18})(1+3h_{18})(1+3h_{18}+3h_{18}^2)(1+6h_{18}+12h_{18}^2)}
{(1 + 6 h_{18} + 6 h_{18}^2)^4}
$$
$$
4\Theta(q)=E_{2,2}+2E_{2,3}-2E_{2,6}-E_{2,9}+E_{2,18},\quad t^3\Theta^4=(\eta_3\eta_6)^8
$$
$$\text{Group: }\Gamma_0(18)+\left<2,9,w\right>\quad\text{where }w=\begin{pmatrix}2 & -1\\6 & -2\end{pmatrix}.$$
$$
\text{Rational CM-values:}\quad t((3+\sqrt{-5})/6)=-\frac14,\quad t((3+\sqrt{-7})/12)=\frac19,
\quad t((3+\sqrt{-7})/6)=-\frac19,
$$
$$
t((1+i)/6)=\frac{1}{12},\quad t(\sqrt{-2}/6)=\frac{1}{16},\quad t(i/3)=\frac{1}{18},
$$
$$
t(\sqrt{-10}/6)=\frac{1}{36},\quad t((3+\sqrt{-13})/6)=-\frac{1}{36},\quad t(\sqrt{-22}/6)=\frac{1}{144},
$$
$$
t(3+5i)/6)=-\frac{1}{180},\quad t((3+\sqrt{-37})/6)=-\frac{1}{576}
$$

\subsection{Ap\'ery$_2$ $\odot$ central}\ \\
Differential operator: 
$(1 - 44 t - 16t^2) \theta^3 + (-66 t - 48 t^2) \theta^2 +(-34 t - 44 t^2) -6 t - 12 t^2$
$$
\text{Coefficients:}\quad g_n=\binom{2n}{n}\sum_{k=0}^n\binom{n}{k}^2\binom{n+k}{k}.
$$

$$\text{Mirror map:}\quad t(q)=\frac{h_5}{1+22h_5+125h_5^2},\quad\text{pole: }(2+i)/5$$
$$-4\Theta(q)=E_{2,5},\quad t\Theta^2=(\eta_1\eta_5)^4,\quad\text{Group: }\Gamma_0(5)+\left<5\right>$$
$$
\text{Rational CM-values:}\quad t((3+\sqrt{-11})/10)=\frac14,\quad t((1+\sqrt{-3/5})/2)=-\frac13,
\quad t((1+\sqrt{-7/5})/2)=-\frac{1}{28},
$$
$$
t((1+\sqrt{-19})/10)=\frac{1}{36},\quad t((-1+2i)/5)=\frac{1}{18},\quad t(\sqrt{-2/5})=\frac{1}{72},
$$
$$
t(\sqrt{-3/5})=\frac{1}{588},\quad t((1+\sqrt{-23/5})/2)=-\frac{1}{828},
\quad t((1+\sqrt{-47/5})/2)=-\frac{1}{15228}
$$

\subsection{Franel $\odot$ central}\ \\
Differential operator: 
$(1 - 28 t - 128 t^2) \theta^3 + (-42 t - 384 t^2) \theta^2 + (-22 t - 352 t^2) \theta -4t - 96 t^2$
$$
\text{Coefficients:}\quad g_n=\binom{2n}{n}\sum_{k=0}^n\binom{n}{k}^3.
$$

$$\text{Mirror map:}\quad t(q)=\frac{h_6+17h_6^2+72h_6^3}{(1+16h_6+72h_6^2)^2},
\quad\text{Group: }\Gamma_0(6)+\left<2,3\right>
$$
$$
\Theta(q)=\frac{(\eta_2\eta_3)^7}{(\eta_1\eta_6)^5} - 
\frac{(\eta_1\eta_6)^7}{(\eta_2\eta_3)^5},\quad t\Theta^2=(\eta_1\eta_2\eta_3\eta_6)^2
$$
$$
\text{Rational CM-values:}\quad t((1+\sqrt{-1/3})/2)=-\frac14,\quad t((1+\sqrt{-5})/6=\frac{1}{16},
\quad t((1+i)/2)=-\frac{1}{16},
$$
$$
t((3+\sqrt{-15})/12)=\frac15,\quad t(\sqrt{-1/6})=\frac{1}{32},\quad t(\sqrt{-1/2})=\frac{1}{96},
$$
$$
t(\sqrt{-1/3})=\frac{1}{50},\quad t(1+\sqrt{-7/3})/2)=-\frac{1}{116},\quad t((1+\sqrt{-5/3})/2)=-\frac{1}{49},
$$
$$
t(\sqrt{-5/6})=\frac{1}{320},\quad t((1+\sqrt{-11/3})/2)=-\frac{1}{400},\quad t(\sqrt{-7/6})=\frac{1}{896},
$$
$$
t((1+\sqrt{-19/3})/2)=-\frac{1}{2704},\quad t(\sqrt{-13/6})=\frac{1}{10400},\quad t(\sqrt{-17/6})=\frac{1}{39200}
$$

\end{document}